\pdfoutput=1
\RequirePackage{ifpdf}
\ifpdf 
\documentclass[pdftex]{sigma}
\else
\documentclass{sigma}
\fi

\DeclareMathOperator{\cl}{C\ell} 
\DeclareMathOperator{\ccl}{\mathbb{C}\ell} 
\DeclareMathOperator{\End}{End}
\DeclareMathOperator{\Op}{Op}
\DeclareMathOperator{\Spin}{Spin}
\DeclareMathOperator{\symb}{symb}
\DeclareMathOperator{\Ind}{Ind}
\DeclareMathOperator{\proj}{proj}
\DeclareMathOperator{\ad}{ad}
\DeclareMathOperator{\id}{id}
\DeclareMathOperator{\Pin}{Pin}
\DeclareMathOperator{\Or}{O}
\DeclareMathOperator{\SO}{SO}

\usepackage{slashed}
\newcommand{\fsl}[1]{{\slashed{#1}}} 

\numberwithin{equation}{section}

\newtheorem{theor}{Theorem}[section]
\newtheorem{prop}[theor]{Proposition}

\newtheorem{lem}[theor]{Lemma}

\begin{document}

\newcommand{\arXivNumber}{2012.09625}

\renewcommand{\PaperNumber}{049}

\FirstPageHeading

\ShortArticleName{Symmetry Breaking Differential Operators for Tensor Products}

\ArticleName{Symmetry Breaking Differential Operators\\ for Tensor Products of Spinorial Representations}

\Author{Jean-Louis CLERC and Khalid KOUFANY}
\AuthorNameForHeading{J.-L.~Clerc and K.~Koufany}
\Address{Universit\'e de Lorraine, CNRS, IECL, F-54000 Nancy, France}
\Email{\href{mailto:jean-louis.clerc@univ-lorraine.fr}{jean-louis.clerc@univ-lorraine.fr}, \href{mailto:khalid.koufany@univ-lorraine.fr}{khalid.koufany@univ-lorraine.fr}}

\ArticleDates{Received January 12, 2021, in final form May 06, 2021; Published online May 13, 2021}

\Abstract{Let $\mathbb S$ be a Clifford module for the complexified Clifford algebra $\ccl(\mathbb R^n)$, $\mathbb S'$ its dual, $\rho$ and $\rho'$ be the corresponding representations of the spin group $\Spin(n)$. The group $G= \Spin(1,n+1)$ is a (twofold) covering of the conformal group of $\mathbb R^n$. For $\lambda, \mu\in \mathbb C$, let $\pi_{\rho, \lambda}$ (resp.~$\pi_{\rho',\mu}$) be the spinorial representation of $G$ realized on a (subspace of) $C^\infty(\mathbb R^n,\mathbb S)$ (resp.~$C^\infty(\mathbb R^n,\mathbb S')$). For $0\leq k\leq n$ and $m\in \mathbb N$, we construct a symmetry breaking differential operator $B_{k;\lambda,\mu}^{(m)}$ from $C^\infty(\mathbb R^n \times \mathbb R^n,\mathbb{S}\,\otimes\, \mathbb{S}')$ into $C^\infty(\mathbb R^n, \Lambda^*_k(\mathbb R^n) \otimes \mathbb{C})$ which intertwines the representations $\pi_{\rho, \lambda}\otimes \pi_{\rho',\mu} $ and $\pi_{\tau^*_k,\lambda+\mu+2m}$, where $\tau^*_k$ is the representation of $\Spin(n)$ on~the space $\Lambda^*_k(\mathbb R^n) \otimes \mathbb{C}$ of complex-valued alternating $k$-forms on $\mathbb{R}^n$.}


\Classification{43A85; 58J70; 33J45}\vspace{-1ex}

\section{Introduction}\vspace{-1ex}

In the last years there had been a lot of work on \emph{symmetry breaking differential operators} (SBDO for short), initiated by a program designed by T. Kobayashi (see~\cite{k} and~\cite{kp16} for more information on the subject). The present authors have already contributed to the construction of some SBDO (see~\cite{bc, bck2, bck1, c19}). In the present paper, we construct such operators in the context of tensor product of two spinorial principal series representations of the conformal spin group of~$\mathbb R^n$ (a two-fold covering of the Lorentz group $\SO_0(1, n+1))$.

The method we follow has been named \emph{source operator method} by the first author (see~\cite{c19} for a systematic presentation). An essential ingredient is the \emph{Knapp--Stein operator} for the spinorial series, which is presented along new lines in the \emph{ambient space} approach (Section~\ref{S33}). The construction of the source operator requires some Fourier analysis on $\mathbb R^n$ and in this paper we develop an approach through an \emph{ad hoc} symbolic calculus, which eases the computations (Section~\ref{S43}). As a result, an explicit expression for the source operator (of degree $4$ with polynomial coefficients) is obtained (see~\eqref{Elambdamu}).

Once the source operator is computed, the sequel is standard and yields a family of constant coefficients bi-differential operators which are covariant for the action of the conformal spinor group. In complement, a recurrence formula on these operators is obtained. An explicit formula is obtained for the \lq\lq simplest\rq\rq\ SBDO (scalar-valued and of degree~2).

When the dimension $n$ is equal to $1$, the operators thus constructed coincide with the classical
\emph{Rankin--Cohen brackets} of even degree, see~\eqref{B_0dim1}.

\section{Clifford modules}\vspace{-1ex}

This section contains what is necessary to know about Clifford algebras and their modules in~order to read this article, without any claim to originality. We use~\cite{dss} and~\cite{bgv} as main references.

\subsection{The Clifford algebra and the spin group}

Let $(E, \langle \cdot\,,\,\cdot\rangle)$ be a Euclidean vector space of dimension $n$. The \emph{Clifford algebra} $\cl(E)$ is the algebra over $\mathbb R$ generated by the vector space $E$ and the relations
\begin{gather*}
xy +yx=-2\langle x,y\rangle\qquad \text{for}\quad x, y \in E,
\end{gather*}
where $-2\langle x,y\rangle$ is identified with $-2\langle x,y\rangle \mathbf{1}$ and $\mathbf{1}$ being the algebra identity element.

There is a natural action of the Clifford algebra $\cl(E)$ on the exterior algebra $\Lambda(E)$. For $x\in E $ and $\omega\in \Lambda(E)$, let $\varepsilon(x)\omega$ be the exterior product of $x$ with $\omega$, and let $\iota(x)\omega$ be the contraction of $\omega$ with the covector $\langle x, \cdot\rangle$. The Clifford action of a vector $x\in E$ on $\Lambda(E)$ is~defined by
\begin{gather*}
c(x) \omega = \varepsilon(x) \omega - \iota(x) \omega.
\end{gather*}
The classical formula
\begin{gather*}\varepsilon(x) \iota(y) + \iota(x) \varepsilon(y) = \langle x,y\rangle
\end{gather*}
implies that
\begin{gather*}
c(x)c(y)+c(y)c(x) = -2\langle x,y\rangle
\end{gather*}
and by the universal property of the Clifford algebra, the action $c$ can be extended to $\cl(E)$.

Associated to this action is the \emph{symbol map} $\sigma \colon \cl(E) \to \Lambda(E)$ given by
\begin{gather*}
\sigma(a) = c(a) \mathbf{1}\qquad \text{for}\quad a \in \cl(E).
\end{gather*}
The symbol map can be shown to be an isomorphism, and its inverse $\gamma \colon \Lambda(E) \to \cl(E)$ is called the \emph{quantization map}, see~\cite{bgv} for more information. For $I=\{ i_1, i_2,\dots,i_k\}$ where $1\leq i_1<i_2<\dots<i_k\leq n$, let $e_I = e_{i_1}\wedge e_{i_2}\wedge\dots\wedge e_{i_k}$. Then the family $\{e_I\}$, where $I$ runs through all possible subsets of $\{1,2,\dots,n\}$ form a basis of $\Lambda(E)$. In this basis the quantization map is given by
\begin{gather}\label{gamma}
\gamma(e_I) =e_{i_1}e_{i_2}\cdots e_{i_k}.
\end{gather}

The orthogonal group $\Or(E)$ acts on E and there is a natural extension of this action both to~$\Lambda(E)$ and to $\cl(E)$. Both $\sigma$ and $\gamma$ are isomorphisms of $\Or(E)$-modules.

 The \emph{conjugation} $\alpha$ is the unique anti-involution of $\cl(E)$
such that for $x\in E$, $\alpha(x) = -x$. Notice that
\begin{gather*}
x\in E,\qquad \vert x\vert = 1 \qquad \text{implies} \qquad \alpha(x) = x^{-1}.
\end{gather*}
The {pin group} $\Pin(E)$ is defined as the multiplicative subset of $\cl(E)$ given by
\begin{gather*}
\Pin(E) = \{g\in \cl(E),\, g= x_1x_2\cdots x_{k},\, \vert x_j\vert = 1 \text{ for } 1\leq j\leq k\} ,
\end{gather*}
the inverse of the element $g=x_1x_2\cdots x_{k}$ being the element
 \begin{gather*}
g^{-1} = \alpha(g) .
\end{gather*}
The {spin group} $\Spin(E)$ is defined similarly as
\begin{gather*}
\Spin(E) = \{g\in \cl(E),\, g= x_1x_2\cdots x_{2k},\, \vert x_j\vert = 1 \text{ for } 1\leq j\leq 2k\}.
\end{gather*}
Let $x\in E$ such that $\vert x\vert = 1$. Then for any $y\in E$, $xyx^{-1}$ belongs to $E$ and
{\sloppy\begin{gather*}
xyx^{-1} = -s_x y,
\end{gather*}
where $s_x$ is the orthogonal symmetry with respect to the hyperplane perpendicular to $x$. As~a~con\-sequence, if $g\in \Pin(E)$, then for any $x\in E$, $gxg^{-1}\in E$ and the map $\tau_g \colon x\mapsto gxg^{-1}$ belongs to~$\Or(E)$. If moreover $g\in \Spin(E)$, then $\tau_g$ belongs to~$\SO(E)$.

}

\begin{prop} The map $g\mapsto \tau_g$ induces homomorphisms
\begin{gather*}
\tau \colon\ \Pin(E) \to \Or(E), \qquad \tau \colon\ \Spin(E) \to\SO(E),
\end{gather*}
 which are twofold coverings.
\end{prop}

\subsection{Clifford module and its dual}

Let $\mathbb E$ be the complexification of $E$ and extend the inner product on $E$ to a symmetric $\mathbb C$-bi\-li\-near form. Denote by $\ccl(E)$ the complex Clifford algebra of $\mathbb E$, which can be identified with $\cl(E)\otimes \mathbb C$.

A \emph{Clifford module} $(\mathbb S, \rho)$ is a complex vector space $\mathbb S$ together with a (left) action $\rho$ of $\ccl(E)$ on $\mathbb S$. By restriction, the action $\rho$ yields representations of the groups $\Pin(E)$ or $\Spin(E)$, also denoted by $\rho$. As $\Pin(E)$ is compact, there exists an inner product $\langle\,\cdot \,,\cdot\rangle$ on $\mathbb S$ for which the action of the group $\Pin(E)$ is unitary. Now for any $v\in E$ and for any $s,t\in \mathbb S$
\begin{gather*}\label{rho(v)}
\langle \rho(v)s, t\rangle= - \langle s, \rho(v)t\rangle .
\end{gather*}
In fact, it suffices to prove the formula for $v\in E$ such that $\vert v\vert =1$. But then $\rho(v)^2 = -1$, so that
\begin{gather*}
\langle \rho(v)s, t\rangle=-\langle \rho(v)s, \rho(v) \rho(v) t\rangle= - \langle s, \rho(v)t\rangle
\end{gather*}
using the unitarity of $\rho(v)$ for $v\in \Pin(E)$.

The dual space $\mathbb S'$ is also a Clifford module with the action $\rho'$ given by
\begin{gather*}
x\in E,\quad t'\in \mathbb S',\qquad \rho'(x)\, t'= -\,t'\circ \rho(x),
\end{gather*}
and then extended to a representation of $\ccl(E)$. The restriction of $\rho'$ to $\Spin(E)$ (still denoted by $\rho'$) coincides with the contragredient representation of $\rho$.

 Denote the duality between $\mathbb S$ and $\mathbb S'$ by $(s, t')$, for $s\in \mathbb S$ and $t'\in \mathbb S'$. Then
\begin{gather}
\forall x\in E,\quad s\in \mathbb S,\quad t'\in\mathbb S', \qquad (\rho(x) s,t') = -(s,\rho'(x)t'),\nonumber
\\[.5ex]
\forall g\in \Spin(E),\quad s\in \mathbb S,\quad t'\in\mathbb S',\qquad (\rho(g) s,\rho'(g)t') = (s,t').\label{1rhorho'1}
\end{gather}

\subsection[Decomposition of the tensor product SxS']
{Decomposition of the tensor product $\boldsymbol{\mathbb S\otimes \mathbb S'}$}

 Let $\Lambda^*(E)$ be the dual of $\Lambda(E)$ and let $\Lambda^*(E)\otimes \mathbb C$ be its complexification, which can be regarded as the space of all complex multilinear alternating forms on $E$.

Define $\Psi \colon \mathbb S\otimes \mathbb S'\to\Lambda^*(E)\otimes \mathbb C$ by the following formula, for $s\in \mathbb S, t'\in \mathbb S'$ and $\omega\in \Lambda(E)$
\begin{gather*}
\Psi(s\otimes t') (\omega) = (\rho\big(\gamma(\omega)\big)s,t')
\end{gather*}
or more explicitly (see~\eqref{gamma}), for $\omega = e_{i_1}\wedge \cdots \wedge e_{i_k}$,
\begin{gather*}
\Psi(s\otimes t') \big(e_{i_1}\wedge \cdots \wedge e_{i_k} \big) = \big( \rho(e_{i_1}) \cdots \rho(e_{i_k})\, s,t'\big),
\end{gather*}
where $1\leq i_1<i_2<\cdots<i_k\leq n$.
\begin{prop} The map $\Psi \colon \mathbb S \otimes \mathbb S' \to \Lambda^*(E) \otimes \mathbb C$ intertwines the representation $\rho\otimes \rho'$ and the natural representation $\tau^*$ of \,$\Spin(E)$ on $\Lambda^*(E)\otimes \mathbb C$.
\end{prop}
\begin{proof}
Let $g\in \Spin(E)$. Recall that for any $x\in E$, $\tau(g) x = gxg^{-1}$, so that
\begin{gather*}
\rho(x) \rho(g) = \rho(g) \rho\big(g^{-1}xg\big) = \rho(g) \rho\big(\tau\big(g^{-1}\big) x)\big)
\end{gather*}
and hence, for $1\leq i_1<i_2<\cdots <i_k\leq n$
\begin{gather*}
\rho(e_{i_1}) \cdots \rho(e_{i_k}) \rho(g) = \rho(g) \rho\big( \tau\big(g^{-1}\big) e_{i_1}\big)\cdots \rho\big( \tau\big(g^{-1}\big) e_{i_k}\big).
\end{gather*}
so that for $s\in \mathbb S$, $t'\in \mathbb S'$
\begin{gather*}
\Psi(\rho(g)s\otimes\rho'(g)t')\big( e_{i_1}\wedge \cdots \wedge e_{i_k}\big) = \big(\rho(e_{i_1}) \cdots \rho(e_{i_k}) \rho(g) s,\rho'(g)t'\big)
\\ \hphantom{\Psi(\rho(g)s\otimes\rho'(g)t')\big( e_{i_1}\wedge \cdots \wedge e_{i_k}\big)}
 {}= \big( \rho(g) \rho\big( \tau\big(g^{-1}\big) e_{i_1}\big)\cdots \rho\big( \tau\big(g^{-1}\big) e_{i_k}\big)s, \rho'(g) t'\big)
 \\ \hphantom{\Psi(\rho(g)s\otimes\rho'(g)t')\big( e_{i_1}\wedge \cdots \wedge e_{i_k}\big)}
{} = \big( \rho\big( \tau\big(g^{-1}\big) e_{i_1}\big)\cdots \rho\big( \tau\big(g^{-1}\big) e_{i_k}\big)s, t'\big)
 \\ \hphantom{\Psi(\rho(g)s\otimes\rho'(g)t')\big( e_{i_1}\wedge \cdots \wedge e_{i_k}\big)}
{} = \tau(g)^*\Psi(s\otimes t') (e_{i_1}\wedge \dots \wedge e_{i_k})
\end{gather*}
we thus get
\begin{gather*}
\Psi(\rho(g)s\otimes \rho'(g)t') = \tau(g)^*\big( \Psi (s\otimes t')\big).
 \tag*{\qed}
\end{gather*}
\renewcommand{\qed}{}
\end{proof}

The space $\Lambda^*(E)\otimes \mathbb C$ decomposes further under the action of the group $\Spin(E)$ (which reduces to an action of $\SO(E)$) and in fact
\begin{gather*}
\Lambda^*(E)\otimes \mathbb C= \bigoplus_{k=0}^n \Lambda^*_k(E)\otimes \mathbb C,
\end{gather*}
where $\Lambda^*_k(E)$ is the space of alternating $k$-forms on $E$.
For $0\leq k\leq n$ let
\begin{gather*}
\Psi^{(k)} \colon\quad \mathbb S\otimes \mathbb S' \to \Lambda^*_k(E)\otimes \mathbb C,\qquad
\Psi^{(k)}= \proj_k\circ \Psi,
\end{gather*}
the operator $\proj_k$ being the projector from $\Lambda^*(E) \otimes \mathbb C$ onto $\Lambda^*_k(E)\otimes \mathbb C$.

The following lemma will be needed in the proof of the next proposition.
\begin{lem}
Let $J=\{j_1,\ldots, j_k \}$ with $1\leq j_1<j_2<\cdots <j_k\leq n$ and let
$e_J = e_{j_1}e_{j_2} \cdots e_{j_k}$. Then
\begin{gather}\label{sumeiei}
 \sum_{i=1}^{n}e_i e_J e_i = (-1)^{k-1} (n-2k) e_J .
\end{gather}
\end{lem}

\begin{proof}
Let $1\leq i\leq n$. Assume first that $i\notin J$. Then
\begin{gather*}
e_ie_Je_i = e_i (e_{j_1} e_{j_2} \cdots e_{j_k}) e_i= (-1)^k (e_{j_1} e_{j_2} \cdots e_{j_k}) e_i e_i= (-1)^{k-1} e_J .
\end{gather*}
Assume on the contrary that $i=j_\ell$ for some $\ell, 1\leq \ell \leq k$. Then
\begin{gather*}
e_i (e_{j_1} e_{j_2} \cdots e_{j_{\ell-1}}) e_{j_\ell} (e_{j_{\ell+1}} \cdots e_{j_k} ) e_i
= (-1)^{\ell-1} (-1)^{k-\ell}(e_{j_1} e_{j_2} \cdots e_{j_{\ell-1}}) e_ie_{j_\ell} e_i(e_{j_{\ell+1}} \cdots e_{j_k} )
\\ \hphantom{e_i (e_{j_1} e_{j_2} \cdots e_{j_{\ell-1}}) e_{j_\ell} (e_{j_{\ell+1}} \cdots e_{j_k} ) e_i}
{}= (-1)^k e_J .
\end{gather*}
Hence
\begin{gather*}
\sum_{i=1}^{n} e_ie_Je_i = (n-k) (-1)^{k-1} e_J+ k(-1)^k e_J
= (n-2k)(-1)^{k-1} e_J . \tag*{\qed}
\end{gather*}
\renewcommand{\qed}{}
\end{proof}

Let $L$ be the operator on $\mathbb S \otimes \mathbb S' $ given by
\begin{gather*}
L(v\otimes w')= \sum_{i=1}^n \rho(e_i)v\otimes \rho'(e_i)w' .
\end{gather*}

\begin{prop} Let $1\leq k\leq n$. Then
\begin{gather}\label{rhorho'3}
\Psi^{(k)} \circ L = (-1)^k (n-2k) \Psi^{(k)} .
\end{gather}
\end{prop}

\begin{proof}
Fix $v\in \mathbb S$ and $w'\in \mathbb S'$. Let $J=\{j_1,\ldots, j_k \}$ with $1\leq j_1<j_2<\cdots <j_k\leq n$ and let $e_J$ be the corresponding $k$-vector.
Then
\begin{gather*}
\Psi^{(k)} \big(L(v\otimes w')\big) (e_J) = \sum_{i=1}^n (\rho(e_{j_1}) \rho(e_{j_2}) \cdots \rho(e_{j_k}) \rho(e_i)v, \rho'(e_i)w'),
\end{gather*}
which by~\eqref{1rhorho'1} is transformed to
\begin{gather*}
\Psi^{(k)} \big(L(v\otimes w')\big) (e_J)=- \sum_{i=1}^n (\rho(e_i) \rho(e_{j_1}) \rho(e_{j_2}) \cdots \rho(e_{j_k}) \rho(e_i)v,w')
\\ \hphantom{\Psi^{(k)} \big(L(v\otimes w')\big) (e_J)}
{}= -\bigg(\rho\bigg(\sum_{i=1}^n e_ie_{j_1} e_{j_2} \cdots e_{j_k} e_i\bigg) v,w'\bigg),
\end{gather*}
and according to~\eqref{sumeiei} we have
\begin{gather*}
\Psi^{(k)} \big(L(v\otimes w')\big) (e_J) = (-1)^k(n-2k) (\rho(e_{j_1} e_{j_2} \cdots e_{j_k}) v,w')
\\ \hphantom{\Psi^{(k)} \big(L(v\otimes w')\big) (e_J)}
{}= (-1)^k (n-2k) \Psi^{(k)}(v\otimes w')(e_J),
\end{gather*}
hence, the conclusion follows.
\end{proof}

\subsection[Spinors and irreducible representations of Spin(E)]
{Spinors and irreducible representations of $\boldsymbol{\Spin(E)}$}

For the sake of completeness, we now discuss the irreducible Clifford modules and the corresponding representations of the spin group, known as \emph{spinor spaces}, see, e.g.,~\cite{dss, deligne, kn}.

When $n$ is even, say $n=2m$, there exists, up to equivalence a unique irreducible Clifford module $\mathbb S_{2m}$ of dimension $2^m$. As a representation of $\Spin(2m)$, $\mathbb S_{2m}$ splits into two irreducible non-equivalent representations, the half spinors spaces $\mathbb S^+_{2m}$ and $\mathbb S^-_{2m}$, each of dimension $2^{m-1}$.

When $n$ is odd, say $n=2m+1$, there exist two non-equivalent irreducible Clifford modules, of dimension $2^m$. As representations of the spin group $\Spin(2m+1)$, they are irreducible and equivalent, thus leading to a unique spinor space $\mathbb S_{2m+1}$.

Whether $n$ is even or odd, the dual of the Clifford module $\mathbb S_n$ is isomorphic to itself as a representation of the spin group $\Spin(n)$. In the even case, the half spinor space is either self dual or isomorphic to its opposite half spinor space, depending on $m$, but in any case,
$\mathbb S_{2m}\otimes \mathbb S_{2m}= \big(\mathbb S_{2m}^+\otimes \mathbb S_{2m}^+\big)\oplus \big(\mathbb S_{2m}^+\otimes \mathbb S_{2m}^-\big)\oplus \big(\mathbb S_{2m}^-\otimes \mathbb S_{2m}^+\big) \oplus \big( \mathbb S_{2m}^-\otimes \mathbb S_{2m}^-\big)$.

{\sloppy
The representation of the spin group $\Spin(n)$ on $\Lambda^*(\mathbb R^n)$ goes down to a representation of~$\SO(n)$ and decomposes as $\bigoplus_{k=0}^n \Lambda^*_k(\mathbb R^n)$. The Hodge operator yields an isomorphism \mbox{$\Lambda^*_k(\mathbb R^n) \simeq \Lambda^*_{n-k}(\mathbb R^n)$}. In~the odd case $\Lambda^*_k(\mathbb R^n)$ is irreducible for any $k$, whereas for $n=2m$, $\Lambda^*_k(\mathbb R^n)$ is irreducible except for $k=m$, and in fact, $\Lambda^*_m\big(\mathbb R^{2m}\big)$ splits in two irreducible non-equivalent representations.

}

 In the present article we chose to work with Clifford modules. The latter considerations show that it is clearly possible to deduce results for spinor or half spinor spaces, just by refining the decomposition under the action of the spin group.

\section{The conformal spin group and the spinorial representations}

In this section, we present the construction of the conformal spin group $\bf G$ of the space $E$, its conformal action on $E$ and the representations of $\bf G$ associated by induction of the Clifford modules. For convenience we identify $E$ with $\mathbb R^n$.

\subsection[The conformal spin group of Rn]{The conformal spin group of $\boldsymbol{\mathbb{R}^n}$}

Let ${\bf E} = \mathbb R^{1,n+1}$ be the real vector space of dimension $n+2$ equipped with the symmetric bilinear form given by
\begin{gather*}
{\bf Q}( {\mathbf x,\mathbf y }) = x_0y_0-x_1y_1-\dots -x_{n+1}y_{n+1} .
\end{gather*}
Denote by $\cl(\bf E)$ the corresponding Clifford algebra, generated by $\mathbf E$ and subject to the relation
\begin{gather*}
\mathbf x\mathbf y+ \mathbf y\mathbf x= 2\mathbf Q(\mathbf x, \mathbf y) .
\end{gather*}
Let $\boldsymbol \alpha$ be the conjugation of the Clifford algebra, i.e., the unique anti-involution of $\cl(\bf E)$ such that $ \boldsymbol \alpha(\mathbf x) = -\mathbf x$ for $\mathbf x\in\mathbf E$.
Let
$ {\mathbf G} = \Spin_0(1,n+1)$ be defined by
\begin{gather*}
\mathbf G=\big\{\mathbf v_1 \mathbf v_2\cdots \mathbf v_{2k},\,k\in \mathbb N,\, \mathbf v_j\in \mathbf E,\, \mathbf Q(\mathbf v_j) = \pm 1,\, \#\{j,\mathbf Q(\mathbf v_j) = -1\} \text{ even} \big\} .
\end{gather*}
Then $ {\mathbf G}$ is a connected Lie group,
 the inverse of an element $\bf g$ is equal to $\bf g^{-1} = \boldsymbol \alpha (g)$. For $\mathbf x\in \mathbf E$ and $\bf g\in \mathbf G$, the element $\bf g\mathbf x \boldsymbol \alpha(\bf g)$ belongs to $\mathbf E$ and the map $\boldsymbol \tau_{\bf g} \colon\mathbf x \mapsto\bf g\mathbf x \boldsymbol \alpha(\bf g)$ defines an~isometry of $(\mathbf E,\mathbf Q)$. Moreover, the map $ \bf g\mapsto\boldsymbol \tau_{\bf g}$ is a Lie group homomporphism from $\mathbf G$ onto $\SO_0({\bf E})\simeq \SO_0(1,n+1)$ which turns out to be a twofold covering (see~\cite{dss} for more details).

The Lie algebra $\mathbf {\mathfrak g}$ of $\mathbf G$ can be realized as the subspace $\cl^2(\mathbf E)$ of bivectors in $\cl(\bf E)$ spanned by $\{e_ie_j,\, 0\leq i<j\leq n+1\}$.
The Lie algebra $\mathfrak g$ is isomorphic to $\mathfrak o(1,n+1)$. The isomorphism of $\mathbf E$ given by
 \begin{gather*}
 e_0 \mapsto e_0,\qquad e_j \mapsto - e_j,\qquad
 1\leq j\leq n+1
 \end{gather*}
 can be extended as an involution of $\cl(\bf E)$ which, when restricted to
$\cl^2(\mathbf E)$ yields a Cartan involution $\theta$ of $\mathfrak g$. The corresponding decomposition of $\mathfrak g$ into eigenspaces of $\theta$ is given by $\mathfrak g = \mathfrak k\oplus \mathfrak s$, where
\begin{gather*}
\mathfrak k = \bigoplus_{1\leq i<j\leq n+1} \mathbb R\, e_ie_j,\qquad
\mathfrak s = \bigoplus_{j=1}^{n+1}\mathbb R\, e_0e_j .
\end{gather*}
A Cartan subspace $\mathfrak a$ of $\mathfrak s$ is given by
\begin{gather*}
\mathfrak a = \mathbb R H,\qquad \text{where}\quad H=e_0e_{n+1} .
\end{gather*}
Now let
\begin{gather*}
\mathfrak m = \sum_{1\leq i<j\leq n} \mathbb R\, e_ie_j,\qquad
\mathfrak n = \bigoplus_{j=1}^n \mathbb R\, e_j(e_0-e_{n+1}),\qquad \overline{\mathfrak n} = \bigoplus_{j=1}^n \mathbb R\, e_j(e_0+e_{n+1}) ,
\end{gather*}
and notice that
\begin{gather*}
[H,\mathfrak m] = 0,\qquad
\ad H_{\vert \mathfrak n} = +2, \qquad
\ad H_{\vert \overline{\mathfrak n}} = -2 .
\end{gather*}
Then
\begin{gather*}
\mathfrak g = \overline{\mathfrak n}\oplus \mathfrak m\oplus \mathfrak a\oplus \mathfrak n
\end{gather*}
is a Gelfand--Naimark decomposition of $\mathfrak g$. By elementary calculation, for $t\in \mathbb R$
\begin{gather*}
a_t:=\exp(te_0e_{n+1}) = \cosh t +\sinh t\, e_0\, e_{n+1} = e_0\big(\cosh t \,e_0+\sinh t \,e_{n+1}\big)
\end{gather*}
and for $y\in \mathbb R^n$
\begin{gather*}
n_y := \exp \frac{y(e_0-e_{n+1})}{2} = 1+\frac{y(e_0-e_{n+1})}{2}
\end{gather*}
and similarly, for $z\in \mathbb R^n$
\begin{gather*}
\overline n_z := \exp \frac{z(e_0+e_{n+1})}{2} = 1+\frac{z(e_0+e_{n+1})}{2} .
\end{gather*}
The analytic Lie subgroups of $\mathbf G$ associated to $\mathfrak a$, $\mathfrak n$, and $\mathfrak {\overline n}$ are isomorphic to their counterparts in $\SO_0(1,n+1)$ and hence are denoted respectively by $A$, $N$, $\overline N$.

The Cartan involution of $\mathfrak g$ can be lifted to a Cartan involution of $\mathbf G$. The fixed point set of~this involution is a maximal compact subgroup
\begin{gather*}
\mathbf K= \bigg\{\mathbf v_1\mathbf v_2\cdots \mathbf v_{2k},\, \mathbf v_j \in \bigoplus_{i=1}^{n+1} \mathbb R e_i,\, Q(\mathbf v_j) = -1, \, 1\leq j \leq 2k\bigg\},
\end{gather*}
isomorphic to $\Spin(n+1)$.
Let $\mathbf M$ be the centralizer of $A$ in $\mathbf K$ which is isomorphic to $\Spin(n)$.
Let $\mathbf M'$ be the normalizer of $A$ in $\mathbf K$. Then the Weyl group $\mathbf M'/\mathbf M$ has two elements. As~a~rep\-resentative of the non-trivial Weyl group element choose
\begin{gather*}
w= e_1e_{n+1}
\end{gather*}
and observe in fact that
\begin{gather*}
wHw^{-1} = \frac{1}{2}e_1e_{n+1}\,e_0 e_{n+1}\, e_{n+1}e_1 = -\frac{1}{2}e_0e_{n+1}= -H.
\end{gather*}

\subsection{The Gelfand--Naimark decomposition}
To the decomposition of $\mathfrak g$ is associated a (partial) decomposition of the group $\mathbf G$, often called the Gelfand--Naimark decomposition.
More precisely, the map
\begin{gather*}
\overline N\times \mathbf M\times A\times N \ni (\overline n, m, a, n) \mapsto \overline n man\in \mathbf G
\end{gather*}
is injective and its image is a dense open subset of full measure in $\mathbf G$. Conversely, let $g\in \mathbf G$ and assume that $g$ belongs to the image. Then there are unique elements $\overline n(g)\in \overline N$, $m(g)\in \mathbf M$, $a(g)\in A$ and $n(g)\in N$ such that
\begin{gather*}
g=\overline n(g) m(g) a(g) n(g) .
\end{gather*}
The following result will be needed in the sequel.

\begin{prop} \label{Gelfanddec}
 Let $x\in \mathbb R^n$, and assume that $x\neq 0$. Let $x' = e_1xe_1$. Then the following iden\-tity holds:
\begin{gather}\label{GN-decomp}
w^{-1} \overline n_x
=\overline n_{\frac{x'}{\vert x\vert^2}} \bigg({-}e_1\frac{x}{\vert x\vert}\bigg) a_{\ln \vert x\vert}\
n_{\frac{x}{\vert x\vert^2}} .
\end{gather}
In particular,
\begin{gather*}
m\big(w^{-1} \overline n_x\big) = -e_1\frac{x}{\vert x\vert},\qquad
\ln a\big(w^{-1}\overline n_x\big)=\ln \vert x\vert .
\end{gather*}
\end{prop}

\begin{proof}
First
\begin{gather*}
w^{-1} \overline n_x= e_{n+1}e_1\bigg(1+\frac{1}{2}x(e_0+e_{n+1})\bigg)
 =-\frac{1}{2}e_1x-\frac{1}{2}e_1xe_0e_{n+1}-e_1e_{n+1} .
\end{gather*}
The right side of the identity~\eqref{GN-decomp} is equal to
\begin{gather*}
 \bigg(\!1+\frac{x'(e_0+e_{n+1})}{2\vert x \vert^2}\bigg)
 \bigg(\!{-}e_1\frac{x}{\vert x\vert}\bigg)\bigg(\frac{1}{2}\bigg(\!\vert x\vert +\frac{1}{\vert x\vert}\bigg) + \frac{1}{2}\bigg(\!\vert x\vert -\frac{1}{\vert x\vert}\bigg)e_0e_{n+1}\bigg)
 \bigg(\!1+\frac{x(e_0-e_{n+1})}{2\vert x \vert^2} \bigg)
\end{gather*}
whereas the left-hand side is obtained by a standard computation, using in particular the fact that $e_1x$ commutes with $e_0 e_{n+1}$ and the relation $x'e_1x = \vert x\vert^2 e_1$.
\end{proof}

 The left action of $\mathbf G$ on $\mathbf G/\mathbf P$ can be transferred to a rational action (not everywhere defined) on $\overline {\mathfrak n} \simeq \mathbb R^n$,
more explicitly,
\begin{gather*}
g(\bar n_x)= \bar n( g\bar n_x),
\end{gather*}
when it is defined and we simply denote the action of $\mathbf G$ on $\mathbb R^n$ by $g(x)$.
 In particular, the action of $\mathbf M$ on $\mathbb R^n$ is given by
\begin{gather*}
m\in \mathbf M,\qquad x\in \mathbb R^n, \qquad x\mapsto mxm^{-1},
\end{gather*}
the action of $A$ is given by
\begin{gather*}
a_t\in A,\qquad x\in \mathbb R^n,\qquad x\mapsto {{\rm e}^{-2t} x},
\end{gather*}
and the action of $\overline N$ is given by
\begin{gather*}
n_v\in \mathbb \overline N,\qquad x\in \mathbb R^n,\qquad x\mapsto x+v.
\end{gather*}

\subsection[The representation induced from a Clifford module and the associated Knapp--Stein operators]
{The representation induced from a Clifford module \\and the associated Knapp--Stein operators}\label{S33}

Let $(\mathbb S, \rho)$ be a Clifford module for the Clifford algebra $\ccl(E)$. The restriction of $\rho$ to the spin group $\mathbf M$ yields a representation of $\mathbf M$, still denoted by $\rho$.

For $\lambda\in \mathbb C$, let $\chi_\lambda$ be the character of $A$ given by
\begin{gather*}
\chi_\lambda(a_t) = {\rm e}^{2t\lambda}\qquad \text {for}\quad t\in \mathbb R .
\end{gather*}
Now consider the representation
 of $\mathbf P= \mathbf M AN$ given by
\begin{gather*}
\rho\otimes \chi_\lambda\otimes 1,
\end{gather*}
and let
\begin{gather*}
\pi_{\rho,\lambda} = \Ind_{\mathbf P}^{\mathbf G} \rho\otimes \chi_\lambda\otimes 1
\end{gather*}
be the associated induced representation from $\mathbf P$ to $\mathbf G$. Let $\mathbb S_{\rho, \lambda}$ be the associated bundle $\bf G\times_{\rho,\lambda} \mathbb S$ over $\bf G/\bf P$ and let $\mathcal H_{\rho, \lambda}$ be the space of smooth sections of $ \mathbb S_{\rho, \lambda}$. The natural action of $\bf G$ on $\mathbb S_{\rho, \lambda}$
gives a realization of $\pi_{\rho, \lambda}$ on $\mathcal H_{\rho, \lambda}$.

Another realization of the representation $\pi_{\rho, \lambda}$, more fitted for calculations is the \emph{non-compact picture}, see \cite[Chapter VII]{kn}. In this model, the representation is given by
\begin{gather*}
\pi_{\rho,\lambda} (g) F (\overline n) = {\chi_\lambda\big(a\big(g^{-1}\overline n\big)\big)^{-1}} \rho\big(m\big(g^{-1}\overline n\big)\big)^{-1} F\left(g^{-1}(\overline n)\right)\!,
\end{gather*}
where $F$ is a smooth $\mathbb S$-valued function on $\overline N$.

Consider now the representation $w\rho$ of $\mathbf M$ defined by
\begin{gather*}
\forall m\in \mathbf M,\qquad (w\rho)(m) = \rho\big(w^{-1} m w\big) .
\end{gather*}
\begin{prop} \label{wrhoeqrho}
The representation $w \rho$ is equivalent to $ \rho$. More precisely, for all $m\in \mathbf M$
\begin{gather}\label{wrho}
 \rho(-e_1)\circ w \rho(m) = \rho(m)\circ \rho(-e_1) .
\end{gather}
\end{prop}
\begin{proof}
Recall that $w=e_1e_{n+1}$ so that for any $m\in \mathbf M$
\begin{gather*}
w\rho(m) = \rho(e_{n+1} e_1me_1 e_{n+1}) .
\end{gather*}
As $e_{n+1}$ anticommutes with $e_1$ and commutes with $m$, this implies
\begin{gather*}
w\rho(m) = \rho(e_1m(-e_1))= \rho(e_1) \rho(m) \rho(-e_1)\end{gather*}
from which~\eqref{wrho} follows by left multiplication by $ \rho(-e_1)= \rho(e_1)^{-1}$.
\end{proof}

Now form the induced representation
\begin{gather*}
\pi_{w \rho, n-\lambda} =\Ind_\mathbf P^\mathbf G (w\rho\otimes \chi_{n-\lambda}\otimes 1) .
\end{gather*}
The Knapp--Stein operators $J_{\rho, \lambda}$ are intertwining operators between $\pi_{ \rho, \lambda}$ and $\pi_{w \rho, n-\lambda}$, see again~\cite{kn} for general information on these operators. Using the equivalence between $w \rho$ and $ \rho$, introduce the operators
\begin{gather*}
I_{ \rho, \lambda} = \rho(-e_1)\circ J_{ \rho, \lambda} .
\end{gather*}

\begin{prop} For any $g\in \mathbf G$,
\begin{gather*}
I_{ \rho,\lambda}\circ \pi_{ \rho,\lambda}(g) = \pi_{ \rho, n-\lambda}(g) \circ I_{ \rho,\lambda} .
\end{gather*}
\end{prop}

\begin{proof}
First, by induction, Proposition~\ref{wrhoeqrho} implies for any $\mu\in \mathbb C$
\begin{gather*}
\rho(-e_1) \circ \pi_{w \rho,\mu}(g) = \pi_{ \rho,\mu}(g)\circ \rho(-e_1) .
\end{gather*}
Hence
\begin{gather*}
I_{\rho,\lambda}\circ\pi_{\rho,\lambda}(g) = \rho(-e_1) \circ J_{\rho,\lambda}\circ \pi_{\rho,\lambda}(g) = \rho(-e_1)\circ \pi_{w \rho, n-\lambda}(g) \circ J_{\rho,\lambda}
\\ \hphantom{I_{ \rho,\lambda}\circ\pi_{\rho,\lambda}(g) }
{}=\pi_{ \rho,n-\lambda}(g) \circ \rho(-e_1)\circ J_{\rho,\lambda} = \pi_{\rho, n-\lambda}(g) \circ I_{\rho,\lambda}. \tag*{\qed}
\end{gather*}
\renewcommand{\qed}{}
\end{proof}

The expression of the corresponding Knapp--Stein operator in the non-compact picture is given by
\begin{gather*}
J_{ \rho,\lambda} F(\overline n_x) = \int_{\mathbb R^n} {\rm e}^{-( 2n-2\lambda)\ln a(w^{-1} \overline n_y)} \rho\big(m\big(w^{-1} \overline n_y\big)
\big) F(\overline n_{x+y})\,{\rm d}y ,
\end{gather*}
which, using Proposition~\ref{Gelfanddec}, can be rewritten more explicitly as
\begin{gather*}
J_{ \rho,\lambda} F(\overline n_x) = \int_{\mathbb R^n} \vert y\vert^{-2n+2\lambda} \rho\bigg({-}e_1\frac{y}{\vert y\vert}\bigg)F(\overline n_{x+y})\,{\rm d}y .
\end{gather*}
In turn, the operator $I_{\rho, \lambda} = \rho(-e_1)\circ J_\lambda$ is given by
\begin{gather*}\label{KS}
 I_{\rho,\lambda} F( \overline n_x) = \int_{\mathbb R^n} \vert y\vert^{-2n+2\lambda} \rho\bigg(\frac{y}{\vert y \vert}\bigg) F(\overline n_{x-y})\,{\rm d}y ,
\end{gather*}
after the change of variables $y\mapsto y'=-y$. The Knapp--Stein operator $I_{ \rho,\lambda}$ is thus shown to be a~convolution operator on $\overline N$, or otherwise said over $\mathbb R^n$. Notice that these operators were already introduced and studied in~\cite{co}.

Now consider simultaneously $\mathbb S$ and its dual $\mathbb S'$. For $\lambda, \mu \in \mathbb C$ the corresponding induced representations are
\begin{gather*}
\pi_\lambda = \Ind_{\bf P}^{\bf G} \rho\otimes\chi_\lambda\otimes 1, \qquad
\pi'_\mu = \Ind_{\bf P}^{\bf G} \rho'\otimes \chi_\mu\otimes 1 .
\end{gather*}
Simplifying the notation, the corresponding intertwining operators	are
\begin{gather}
I_\lambda f(x) = \int_{\mathbb R^n}\vert y\vert^{-2n+2\lambda}\rho\bigg( \frac{y}{\vert y\vert}\bigg) f(x-y)\,{\rm d}y , \nonumber
\\[1ex]
I_\mu' f(x) = \int_{\mathbb R^n} \vert y\vert^{-2n+2\mu}\rho' \bigg( \frac{y}{\vert y\vert}\bigg)f(x-y)\,{\rm d}y .\label{defKS}
\end{gather}

Finally, consider the \lq\lq outer\rq\rq\ tensor product $\rho\otimes \rho'$ as a representation of $\bf M\times \bf M$ and, for~$\lambda, \mu\in \mathbb C$ form the tensor product representation
\begin{gather*}
\pi_\lambda \otimes \pi'_\mu= \Ind_{\bf P\times \bf P}^{\bf G\times \bf G}(\rho\otimes\chi_\lambda\otimes 1)\otimes (\rho'\otimes \chi_\mu\otimes 1) .
\end{gather*}

\begin{prop}
The operator $I_\lambda \otimes I'_\mu$ intertwines the representations $\pi_\lambda \otimes \pi'_\mu$ and $\pi_{n-\lambda} \otimes \pi'_{n-\mu}$ of $\bf G\times \bf G$.
\end{prop}
The diagonal subgroup of $\bf G \times \bf G$ will be denoted simply by $\bf G$, and viewed as acting diagonally on $\mathbb R^n\times \mathbb R^n$. Needless to say, the previous proposition implies that $I_\lambda \otimes I'_\mu$ is an intertwining operator for the action of $\bf G$ on $C^\infty(\mathbb R^n\times \mathbb R^n, \mathbb S\otimes\mathbb S')$ by the \lq\lq inner\rq\rq\ tensor product $\pi_\lambda\otimes \pi'_\mu$.

\section{The source operator}

\subsection{Definition of the source operator and the main theorem}

Let $\mathcal M$ be the operator on $C^\infty(\mathbb R^n\times \mathbb R^n, \mathbb S\otimes\mathbb S')$ defined for $F$ a smooth function on $\mathbb R^n\times \mathbb R^n$ with values in $\mathbb S\otimes \mathbb S'$ by
\begin{gather*}
\mathcal{M} F(x,y) = \vert x-y\vert^2 F(x,y) .
\end{gather*}

\begin{prop} \label{covM}
 Let $\lambda, \mu\in \mathbb C$. Then for any $g\in \bf G$,
 \begin{gather*}
 \mathcal{M} \circ\big(\pi_\lambda(g) \otimes \pi'_\mu(g)\big) = \big(\pi_{\lambda-1}(g)\otimes \pi'_{\mu-1}(g)\big)\circ \mathcal{M} .
 \end{gather*}
\end{prop}

\begin{proof} The result is a consequence of the following \emph{covariance property} of the function $\vert x-y\vert^2$ under a conformal transformation $g\in \bf G$
\begin{gather}\label{cov1}
\vert g(x)-g(y)\vert^2 = {\rm e}^{-2\ln a(g,\, x)} \vert x-y\vert^2 {\rm e}^{-2\ln a(g,\, y)} ,
\end{gather}
where $a(g,x) = a(g\bar{n}_x)$. This is equivalent to the more classical formula
\begin{gather}\label{cov2}
\vert g(x)-g(y)\vert^2 = \kappa(g,x)\, \vert x-y\vert^2\, \kappa(g,y),
\end{gather}
where $\kappa(g,x)$ stands for the conformal factor of $g$ at $x$. The equivalence of~\eqref{cov1} and~\eqref{cov2} comes from the relation
\begin{gather*}
\kappa(g,x) = \chi_1\big(a(g,x) ^{-1}\big) = {\rm e}^{-2\ln a(g,x)}.
\end{gather*}
 The group $\mathbf G$ is generated by $\overline N$, $\mathbf M$, $A$ and $w$. The equality is easy to verify for $g\in \overline N$, $\mathbf M$, and $A$ and follows for $w$ from the Gelfand--Naimark decomposition obtained in Proposition~\ref{Gelfanddec}.
 The proof of the intertwining property is then straightforward. In fact, let $g\in \mathbf G$ and $F\in C^\infty(\mathbb R^n\times \mathbb R^n, \mathbb S\otimes \mathbb S')$. Then
 \begin{gather*}
 \mathcal{M}\circ \big(\pi_\lambda(g)\otimes \pi'_\mu(g)\big)F (x,y)
= \vert x-y\vert^2 \mathrm{e}^{-2\lambda \ln a(g^{-1},\,x)} \mathrm{e}^{-2\mu \ln a(g^{-1},\,y)}
\\ \hphantom{ \mathcal{M}\circ \big(\pi_\lambda(g)\otimes \pi'_\mu(g)\big)F (x,y)=}
{}\times \rho\big(m\big(g^{-1} \overline n_x\big)\big)\otimes \rho\big(m\big(g^{-1}\overline n_y\big)\big) F\big(g^{-1}(x), g^{-1}(y)\big)
 \end{gather*}
 and by using~\eqref{cov1} this can be transformed as
\begin{gather*}
\mathcal{M}\circ \big(\pi_\lambda(g)\otimes \pi'_\mu(g)\big)F (x,y)
= \big\vert g^{-1}(x) -g^{-1}(y) \big\vert^2 \mathrm{e}^{-2(\lambda-1) \ln a(g^{-1},\,x)}\mathrm{e}^{-2(\mu-1) \ln a(g^{-1},\,y)}
\\ \hphantom{\mathcal{M}\circ \big(\pi_\lambda(g)\otimes \pi'_\mu(g)\big)F (x,y)= }
{}\times \rho\big(m\big(g^{-1} \overline n_x\big)\big)
\otimes \rho\big(m\big(g^{-1}\overline n_y\big)\big)F\big(g^{-1}(x), g^{-1}(y)\big)
\\ \hphantom{\mathcal{M}\circ \big(\pi_\lambda(g)\otimes \pi'_\mu(g)\big)F (x,y)}
{} = (\pi_{\lambda-1}(g) \otimes \pi_{\mu-1}(g)) \mathcal{M} F(x,y) .
 \tag*{\qed}
\end{gather*}
\renewcommand{\qed}{}
\end{proof}

 There is a version of these results in the compact picture. As we are mostly interested in~the non-compact picture, we only sketch the argument and refer to \cite[Proposition 1.1]{bc} for more details. The space $\mathbf G/\mathbf P$ can be identified with the Euclidean unit sphere $S^n$ in $\mathbb R^{n+1}$, the classical stereographic projection from $S^n$ into $\mathbb R^n$ corresponds to the map $gP \mapsto \overline n(g)$. The~group~$\mathbf G$ acts conformally on $S^n$, the function $\vert \widetilde x-\widetilde y\vert_{\mathbb R^{n+1}}^2$ defined for $( \widetilde x,\widetilde y)\in S^n\times S^n$ satisfies a covariance relation under the action of $\mathbf G$, similar to~\eqref{cov2}, where $\kappa(g,x)$ is now replaced by the conformal factor of $g$ at $\widetilde x$. For any $\lambda, \mu\in \mathbb C$, the multiplication by the function $\vert \widetilde x-\widetilde y\vert_{\mathbb R^{n+1}}^2$ induces an~ope\-ra\-tor $\mathcal M_{\lambda, \mu}\colon \mathcal H_\lambda \otimes \mathcal H'_\mu\to \mathcal H_{\lambda-1} \otimes \mathcal H'_{\mu-1}$
 which intertwines $\pi_\lambda\otimes \pi'_{\mu}$ and $\pi_{\lambda-1}\otimes \pi'_{\mu-1}$ and is expressed in the local chart $\overline N\to \bf G/ \bf P$ by the operator~$\mathcal{M}$.

 Now let us consider the operator $E_{\lambda, \mu}$ defined by the following diagram
$$
 \begin{CD}
 \mathcal H_{\lambda}\otimes \mathcal H_{\mu}@>E_{\lambda,\mu} >>\mathcal H_{\lambda+1}\otimes \mathcal H_{\mu+1}\\
 @VI_\lambda\otimes I'_\mu V V
 @AA I_{n-\lambda-1}\otimes I'_{n-\mu-1}A\\
 \mathcal H_{n-\lambda}\otimes \mathcal H_{n-\mu}@>\mathcal{M}_{\lambda, \mu}> >\mathcal H_{n-\lambda-1} \otimes \mathcal H_{n-\mu-1}\\
 \end{CD}
$$

\begin{theor}\label{maintheorem}
 The operator $E_{\lambda, \mu}$ is a \emph{differential operator} on the bundle $\mathbb S_{\rho,\lambda} \times \mathbb S_{\rho'\!,\mu}$, which satisfies, for any $g\in \bf G$
 \begin{gather*}
 E_{\lambda, \mu} \circ \big(\pi_\lambda(g)\otimes \pi'_\mu(g)\big) =\big(\pi_{\lambda+1}(g) \otimes \pi'_{\mu+1}(g)\big)\circ E_{\lambda, \mu}.
 \end{gather*}
\end{theor}

 This is the main theorem of the article. The operator is named the \emph{source operator} as it is the key to the construction of the SBDO as we will show later. The fact that $E_{\lambda,\mu}$ is $\bf G$-intertwining is a consequence of the definition. The fact that it is a differential operator is much more subtle and will be shown by working in the non-compact picture. There is however some difficulty when using the non-compat picture, due to the fact that the space of $C^\infty$ vectors in~the non-compact picture is not very manageable, especially when using the Fourier transform on~$\mathbb R^n$. Hence we~have to use a slightly different path to construct and explicitly calculate the local expression of~the operator $E_{\lambda, \mu}$. Coming back to the compact picture, for generic $\lambda$ (resp.~$\mu$) the Knapp--Stein operator $I_\lambda$ is invertible, and up to a non-zero constant multiple, its inverse is equal to $I_{n-\lambda}$. So~that the operator $E_{\lambda, \mu}$ (up to constant multiple) satisfies the relation
\begin{gather*}
\big(I_{n-\lambda-1} \otimes I_{n-\mu-1}\big) \circ \mathcal{M} = E_{\lambda, \mu} \circ \big(I_{n-\lambda} \otimes I_{ n-\mu}\big) .
\end{gather*}
This is the way we will introduce and calculate the expression of the source operator in the non-compact picture (cf.~Section~\ref{S44}).

\subsection{Riesz distributions for Clifford modules}\label{S42}
Up to this point, the intertwining operators are only formally defined and we need to look more carefully to the convolution kernels of the Knapp--Stein operators.

First recall the classical \emph{Riesz distributions}. For $s\in \mathbb C$, the \emph{Riesz distribution} $r_s$ on $\mathbb R^n$ is given by
\begin{gather*}
r_s(x) = \vert x\vert^s .
\end{gather*}
More precisely, for $\Re(s)>-n$, the function $r_s$ is locally integrable and has moderate growth at infinity, so that $r_s$ is a well-defined tempered distribution. The family of distributions thus defined can be extended analytically in the parameter $s\in \mathbb C$, with poles at $ -n-2k$, $k\in \mathbb N$.

Let $(\mathbb S, \rho)$ be a Clifford module and for $s\in \mathbb C$ define the associated \emph{Clifford--Riesz distribu\-tion}~by
\begin{gather}\label{CliffordRiesz}
\fsl{r}_s(x) = \vert x\vert^s \rho\bigg(\frac{x}{\vert x\vert}\bigg)= \vert x\vert^{s-1} \rho(x) .
\end{gather}
Let $E_j = \rho(e_j)$, $1\leq j\leq n$. Then, for $x=\sum_{j=1}^n x_je_j$, $\rho(x) = \sum_{j=1}^n x_j E_j$. Use the identity
\begin{gather*}
x_j \vert x\vert^{s-1} = \frac{1}{s+1} \frac{\partial}{\partial x_j} \big(\vert x\vert^{s+1}\big) ,
\end{gather*}
to conclude that
\begin{gather*}
\fsl{r}_s(x) = \frac{1}{s+1} \sum_{j=1}^n \frac{\partial \,r_{s+1}}{\partial x_j} (x) E_j .
\end{gather*}
From this expression it is easy to deduce the next statement.
\begin{prop} The family $\fsl{r}_s$ defined by~\eqref{CliffordRiesz} is a meromorphic family of $\End(\mathbb S)$-valued tempered distributions with poles at $s=-n-1-2k$, $k\in \mathbb N$.
\end{prop}
Further properties of these distributions will be needed in the sequel. Parts of the present results were already obtained in~\cite{co} and in~\cite{fos}.
\begin{prop}
\begin{gather}
\partial_j\, \fsl{r}_s(x) = \big((s-1)x_j-\rho(e_jx)\big) \fsl{r}_{s-2}(x),\label{partial1}
\\
\Delta \fsl{r}_s(x) = (s-1)(s+n-1)\, \fsl{r}_{s-2}(x).\label{partial2}
\end{gather}

\end{prop}
\begin{proof} First
\begin{gather*}
\partial_j \left(\vert x\vert^{s-1} \rho(x)\right) = \left((s-1) x_j \vert x\vert^{s-3}\right)\rho(x)+ \vert x\vert^{s-1} \rho(e_j)
\\ \hphantom{\partial_j \left(\vert x\vert^{s-1} \rho(x)\right) }
{}= \left((s-1) x_j \vert x\vert^{s-3}\right)\rho(x) - \vert x\vert^{s-3} \rho(e_j)\rho(x) \rho(x)
\\ \hphantom{\partial_j \left(\vert x\vert^{s-1} \rho(x)\right) }
{}= \big((s-1)x_j-\rho(e_jx)) \vert x\vert^{s-3}\rho(x) ,
\end{gather*}
and~\eqref{partial1} follows.

Next, using~\eqref{partial1}
\begin{gather*}
\partial_j^2\fsl{r}_s(x) = \big((s\!-1)\!-\rho(e_je_j)\big) \vert x\vert^{s\!-3}\rho(x)
\!+\!\big((s\!-1)x_j\!-\rho(e_jx)\big) \big((s\!-3) x_j\!-\rho(e_jx)\big) \vert x\vert^{s\!-5} \rho(x)
\\ \hphantom{\partial_j^2\fsl{r}_s(x)}
{}= s\vert x \vert^{s-3} \rho(x) + \big( (s-1)(s-3) x_j^2 -(2s-4)x_j\rho(e_jx)+\rho(e_jxe_jx)\big)\vert x\vert^{s-5}\rho(x) .
\end{gather*}
Now sum over $j$ from $j=1$ to $j=n$ and use that
\begin{gather*}
\sum_{j=1}^n x_j^2 = \vert x\vert^2, \qquad
\sum_{j=1}^n x_je_j = x,\qquad
\sum_{j=1}^n e_jxe_j = (n-2) x
\end{gather*}
to get
\begin{gather*}
\Delta \fsl{r}_s(x) = \big(ns\vert x \vert^2+(s-1)(s-3) \vert x\vert^2+(2s-4)\vert x\vert^2-(n-2) \vert x\vert^2\big)\vert x \vert^{s-5}\rho(x)
\\ \hphantom{\Delta \fsl{r}_s(x)}
{}=\big(s^2+(n-2)s-n+1\big) \vert x\vert^{s-3}\rho(x) ,
\end{gather*}
and~\eqref{partial2} follows.
\end{proof}

We will also need the Fourier transform of the Riesz distributions. The Fourier transform of a function $f$ on $E$ is defined by the formula
\begin{gather*}
\mathcal Ff(\xi) = \widehat f(\xi) = \int_E {\rm e}^{-{\rm i}\langle x,\xi\rangle} f(x)\,{\rm d}x.
\end{gather*}
The Fourier transform is an isomorphism of $\mathcal S(E)$ onto $\mathcal S(E)$, the definition of the Fourier transform can be extended by duality to the space of tempered distributions $\mathcal S'(E)$. For $V$ a
 finite-dimensional vector space, denote by $\mathcal S(E,V)$ (resp.~$\mathcal S'(E, V)$) the space of
 $V$--valued Schwartz functions (resp.~tempered distributions). The Fourier transform can also be extended to these spaces.

Recall the following classical result for the usual Riesz distributions, see, e.g.,~\cite{gs}.
\begin{prop}
The Fourier transform of the Riesz distribution $r_s$ is given by
\begin{gather*}
\mathcal F (r_s) (\xi) = c_s r_{-s-n}(\xi) ,
\end{gather*}
where $\displaystyle c_s= 2^{s+n} \pi^{\frac{n}{2}} \frac{\Gamma\big(\frac{s+n}{2}\big)}{\Gamma\big({-}\frac{s}{2}\big)}$.
\end{prop}

\begin{prop} The Fourier transform of $\fsl{r}_s$ is given by
\begin{gather}\label{Fourierspin}
\mathcal F \fsl{r}_s = \fsl{c}_s \fsl{r}_{-s-n},
\end{gather}
where
$
\displaystyle\fsl{c}_s= -{\rm i}2^{s+n}\pi^{\frac{n}{2}} \frac{\Gamma\big(\frac{s+n+1}{2}\big)}{\Gamma\big({-}\frac{s-1}{2}\big)} .
$
\end{prop}
\begin{proof}
For any $j$, $1\leq j\leq n$, a basic formula for the Fourier transform yields
\begin{gather*}
\mathcal F\big(x_j \vert x\vert^{s-1}\big)(\xi) = {\rm i}\frac{\partial}{\partial \xi_j} \big(\mathcal F \vert x\vert^{s-1}\big)(\xi)
= {\rm i} c_{s-1} \frac{\partial}{\partial \xi_j} \big(\vert \xi\vert^{-s+1-n}\big)
\\ \hphantom{\mathcal F\big(x_j \vert x\vert^{s-1}\big)(\xi)}
{}= {\rm i}(-s+1-n) c_{s-1} \xi_j\vert \xi\vert^{-s-1-n}= \fsl{c}_s \xi_j \vert \xi\vert^{-s-n-1}
\end{gather*}
and~\eqref{Fourierspin} follows easily by using the linearity of $x\mapsto \rho(x)$.
\end{proof}

\subsection{A symbolic calculus}\label{S43}
This section describes in a general context a symbolic calculus, inspired by the calculus for the Weyl algebra or of the pseudo-differential calculus, but designed for our specific problems to be treated in the next section. In particular, classical pseudo-differential calculus requires regularity of the symbol in the cotangent variable $\xi$, whereas we have to handle symbols which are homogeneous in $\xi$, and hence not necessarily smooth (even singular) at $0$.

Let $E$ be a Euclidean vector space of dimension $n$, and $V$ be a finite-dimensional vector space.

Let $k\in \mathcal S'(E,\End(V))$. Then, for $f\in \mathcal S(E,V)$, the formula
\begin{gather*}
Kf(x) = \int_E k(x-y) f(y)\,{\rm d}y
\end{gather*}
defines a convolution operator $K$ which maps $\mathcal S(E,V)$ into $\mathcal S'(E,V)$.

As in the scalar case, these operators have a nice version through the Fourier transform, namely
\begin{gather*}
\widehat {Kf} (\xi) = \widehat k (\xi) \widehat f(\xi), \qquad \xi\in E ,
\end{gather*}
where $\widehat k \in \mathcal S'(E',V)$ is the Fourier transform of the distribution $k$.

Let $p(x)$ be an $\End(V)$-valued polynomial function on $E$. Then, for $f\in \mathcal S(E,V)$ the formula
\begin{gather*}
f \mapsto \big(x\mapsto p(x) f(x)\big)
\end{gather*}
defines an operator on $\mathcal S(E,V)$
 denoted by $f\mapsto pf$ and referred to as the \emph{multiplication operator} by
$p$. A multiplication operator can be extended to $\mathcal S'(E,V)$.

We will have to deal with operators from $\mathcal S(E,V)$ into $\mathcal S'(E,V)$ which are obtained by composing a convolution operator (say $K$) followed by a multiplication operator by an $\End(V)$-valued polynomial function (say $p$) on $E$. Such an operator will be denoted by $p\,K$. By definition, its \emph{symbol} is given by
\begin{gather*}
\symb(p\,K) (x,\xi) = p(x)\circ \widehat k(\xi),\qquad x\in E,\qquad \xi\in E'
\end{gather*}
viewed as the polynomial function on $E$ with values in $\mathcal S'(E,\End(V))$
\begin{gather*}
x\mapsto p(x)\circ \widehat k(\xi) .
\end{gather*}
We let $\Op(E,V)$ be the family of finite linear combinations of such operators. In other words, an element of $\Op(E,V)$ can be written in a unique way as
\begin{gather*}
\sum_\alpha x^\alpha A_\alpha K_\alpha,
\end{gather*}
where $\alpha=(\alpha_1,\alpha_2,\dots, \alpha_n)$ denotes an $n$-multi-index, $A_\alpha$ in $\End(V)$ and $K_\alpha$ is a convolution operator by a tempered $\End(V)$-valued distribution on $E$, with the tacit convention that only a~finite number of terms in the sum are non-zero. Then the symbol of such an operator is given~by
\begin{gather*}
\sum_\alpha x^\alpha A_\alpha \widehat k_\alpha(\xi) .
\end{gather*}
A constant coefficient $\End(V)$-valued differential operator on $E$ is an example of a convolution operator with a tempered distribution, namely a combination of derivatives of the Dirac distribution at $0\in E$.
In particular $\Op(E,V)$ contains the $\End(V)$-valued Weyl algebra on $E$, denoted by $\mathcal W(E,V)$, consisting of the differential operators on $E$ with $\End(V)$-valued polynomial coef\-ficients. Notice that these operators map $\mathcal S(E,V)$ into $\mathcal S(E,V)$ and $\mathcal S'(E,V)$ into $\mathcal S'(E,V)$. Recall the usual definition of the symbol of a differential operator, namely the $\End(V)$-valued polynomial function $\sigma_D$ on $E\times E'$ is given
\begin{gather*}
D {\rm e}^{{\rm i}\langle x,\xi\rangle} = \sigma_D(x,\xi){\rm e}^{{\rm i}\langle x,\xi\rangle} .
\end{gather*}
Then an elementary computation shows that $\sigma_D$ coincides with $\symb(D)$ (see~\cite{c19} for the scalar case). More explicitly, let $D=\sum_\alpha p_\alpha(x) \partial_x^\alpha$ be in $\mathcal W(E,V)$, where $p_\alpha$ is an $\End(V)$-valued polynomial. Then its symbol is given by
\begin{gather*}
\symb(D)(x,\xi) = \sum_\alpha p_\alpha(x)({\rm i}\xi)^\alpha .
\end{gather*}
Although $\Op(E,V)$ is not an algebra of operators, some compositions are possible. Although a~more general result could be stated, we consider only two cases, which will be enough for the present paper.

\begin{prop} 
 Let $D$ be an $\End(V)$-valued differential operator on $E$, and let $K$ be a~convolution operator with a tempered distribution. Then $D\circ K$ belongs to $\Op(E, V)$, and its symbol is given by
\begin{gather}\label{symb1}
 \symb(D\circ K)(x,\xi) = \symb(D) (x,\xi) \circ \symb(K)(\xi) .
\end{gather}
\end{prop}
\begin{proof} Let $\alpha$ be an $n$-multi-index and let $k$ be the kernel of the convolution operator $K$. Then
$\partial_x^\alpha \circ K$ is the convolution operator with kernel $\partial_x^\alpha k(x)$. Hence
\begin{gather*}
\symb(\partial_x^\alpha \circ K) (\xi) = ({\rm i}\xi)^\alpha \widehat k(\xi) ,
\end{gather*}
which coincides with formula~\eqref{symb1} for this particular case. The general formula follows easily.
\end{proof}

\begin{prop} 
 Let $p$ be an scalar-valued polynomial on $E$, and let $L$ be in $\Op(E,V)$. Then $L\circ p$ belongs to $\Op(E,V)$ and its symbol is given by
\begin{gather}\label{symb2}
\symb(L\circ p)(x,\xi) = \sum_\alpha\frac{1}{\alpha !}\, \partial_x^\alpha p\,(x) \bigg( \frac{1}{\rm i} \partial_\xi\bigg)^\alpha \symb(L)(x,\xi) .
\end{gather}
\end{prop}
\begin{proof} Assume first that $L=K$ is a convolution operator by a tempered distribution $k$. Then the symbol of the composition $K\circ p$ can be computed exactly as in the scalar case (see~\cite[Proposition~1.2]{c19}) and the result coincides with~\eqref{symb2}, due to the fact that we assume that $p$ is a~scalar-valued polynomial. The general case follows easily.
\end{proof}

\subsection{The main formula}\label{S44}

We now apply the symbolic calculus developed in the previous section to the construction of the source operator in the non-compact picture. In particular, we come back to the context and notation of Section~\ref{S42}.

For $s,t\in \mathbb C$, consider the \emph{normalized Clifford--Riesz convolution operators}
\begin{gather}
\fsl{\mathcal R}_s \colon\quad \mathcal S(\mathbb R^n, \mathbb S)\to \mathcal S'(\mathbb R^n, \mathbb S), \qquad
\fsl{\mathcal R}_s f = \fsl{c}_s^{-1}\fsl{r}_s \star f ,\nonumber
\\
\fsl{\mathcal R}'_t \colon\quad \mathcal S(\mathbb R^n, \mathbb S')\to \mathcal S'(\mathbb R^n, \mathbb S' ), \qquad
\fsl{\mathcal R}'_t f = \fsl{c}_t^{-1}\fsl{r}'_t \star f .\label{defCR}
\end{gather}
The technical reason for this normalization is that
\begin{gather*}
\symb(\fsl{\mathcal R}_s)(\xi) = \fsl{r}_{-n-s}(\xi)\, ,\qquad
\symb(\fsl{\mathcal R}'_t)(\zeta) = \fsl{r}'_{-n-t}(\zeta) ,
\end{gather*}
see~\eqref{Fourierspin}.

The family $\fsl{\mathcal R}_s$ depends meromorphically on the parameter $s$. Poles and residues were studied in~\cite{co}. An example of the residues is the \emph{Dirac operators}, which in our context comes in two versions, given by
\begin{gather*}\label{defDirac}
\fsl{D} = \sum_{j=1} \rho(e_j) \frac{\partial}{\partial x_j}, \qquad \fsl{D}' = \sum_{j=1} \rho'(e_j) \frac{\partial}{\partial x_j} .
\end{gather*}
As a consequence, they satisfy an intertwining property under the action of the conformal group, and similar results are valid for their powers, see~\cite{co}.

Further, consider the operator
\begin{gather*}
\fsl{\mathcal{M}}_{s,t} \colon\quad \mathcal S(\mathbb R^n\times \mathbb R^n, \mathbb S\otimes \mathbb S')\to \mathcal S'(\mathbb R^n\times \mathbb R^n, \mathbb S\otimes \mathbb S')
,\qquad \fsl{\mathcal{M}}_{s,t} = \big(\fsl{\mathcal R}_s\otimes \fsl{\mathcal R}'_t\big)\circ \mathcal M .
\end{gather*}
which clearly belongs to $\Op\big(\mathbb R^n\times \mathbb R^n, \mathrm{End}(\mathbb{S}\, \otimes\, \mathbb{S}')\big)$ as considered in Section~\ref{S43}.

\begin{prop} \label{symb-F-s-t} We have,
\begin{gather*}
\symb\,(\fsl{\mathcal M}_{s,t})(x,y,\xi,\zeta)= f_{s,t}(x,y, \xi,\zeta) \circ \big( \fsl{r}_{-s-n-2}(\xi) \otimes \fsl{r}'_{-t-n-2}(\zeta)\big),
\end{gather*}
where
\begin{gather*}
f_{s,t}(x,y,\xi,\zeta) = \vert x-y\vert^2 \vert \xi\vert^2 \otimes \vert \zeta\vert^2
 +2{\rm i}(s+n+1) \sum_{j=1}^n (x_j-y_j)\xi_j\otimes \vert \zeta\vert^2
 \\ \hphantom{f_{s,t}(x,y,\xi,\zeta)=}
{} +2{\rm i}(t+n+1) \sum_{j=1}^n (y_j-x_j) \vert \xi\vert^2 \otimes \zeta_j
 +2{\rm i} \big(\rho(x-y)\rho(\xi)\otimes \vert \zeta\vert^2\big)
 \\ \hphantom{f_{s,t}(x,y,\xi,\zeta)=}
 {}+2{\rm i} \big(\vert \xi\vert^2 \otimes \rho'(y-x)\rho'(\zeta)\big)
 -(s+1)(s+n+1) \id\otimes \vert \zeta\vert^2
 \\ \hphantom{f_{s,t}(x,y,\xi,\zeta)=}
{} -(t+1)(t+n+1) \vert \xi\vert^2 \otimes \id
 +2(s+n+1)(t+n+1) \sum_{j=1}^n \xi_j\otimes \zeta_j
 \\ \hphantom{f_{s,t}(x,y,\xi,\zeta)=}
{} +2(s+n+1) \sum_{j=1}^n \xi_j\otimes \rho'(e_j)\rho'(\zeta)
+2(t+n+1)\sum_{j=1}^n \rho(e_j)\rho(\xi)\otimes \zeta_j
\\ \hphantom{f_{s,t}(x,y,\xi,\zeta)=}
{} +2\sum_{j=1}^n \rho(e_j)\rho(\xi)\otimes \rho'(e_j)\rho'(\zeta) .
\end{gather*}
\end{prop}
\begin{proof}
First,
\begin{gather*}
\symb(\fsl{\mathcal R}_s\otimes \fsl{\mathcal R}'_t)(\xi,\zeta) = \fsl{r}_{-s-n} (\xi)\otimes \fsl{r}'_{-t-n}(\zeta) .
\end{gather*}
Following~\eqref{symb2} the composition formula for the symbols yields
\begin{gather*}
\symb\big(\fsl{\mathcal R}_s\otimes \fsl{\mathcal R}'_t\circ \vert x-y\vert^2\big)
 = \vert x-y\vert^2\fsl{r}_{-s-n}(\xi)\otimes \fsl{r}'_{-t-n}(\zeta)
\\ \hphantom{\symb\big(\fsl{\mathcal R}_s\otimes \fsl{\mathcal R}'_t\circ \vert x-y\vert^2\big) =}
{}+ 2\sum_{j=1}^n (x_j-y_j)\bigg(\frac{1}{\rm i}\frac{\partial}{\partial \xi_j}\bigg) \fsl{r}_{-s-n}(\xi)\otimes \fsl{r}'_{-t-n}(\zeta)
\\ \hphantom{\symb\big(\fsl{\mathcal R}_s\otimes \fsl{\mathcal R}'_t\circ \vert x-y\vert^2\big) =}
{}+2\sum_{j=1}^n (y_j-x_j) \fsl{r}_{-s-n}(\xi)\otimes \bigg(\frac{1}{\rm i}\frac{\partial}{\partial \zeta_j} \bigg)\fsl{r}'_{-t-n}(\zeta)
\\ \hphantom{\symb\big(\fsl{\mathcal R}_s\otimes \fsl{\mathcal R}'_t\circ \vert x-y\vert^2\big) =}
{}-\Delta \fsl{r}_{-s-n}(\xi)\otimes \fsl{r}'_{-t-n}(\zeta)
\\ \hphantom{\symb\big(\fsl{\mathcal R}_s\otimes \fsl{\mathcal R}'_t\circ \vert x-y\vert^2\big) =}
{}+2\sum_{j=1}^n \partial_j\fsl{r}_{-s-n}(\xi)\otimes \partial_j \fsl{r}'_{-t-n}(\zeta)- \fsl{r}_{-s-n}(\xi)\otimes \Delta \fsl{r}'_{-t-n}(\zeta).
\end{gather*}
Now use formulas~\eqref{partial1} and~\eqref{partial2} applied to $\fsl{r} $ and $ \fsl{r}'$ to get the result.
\end{proof}

Notice that $f_{s,t}$ is the symbol of a differential operator on $\mathbb R^n\times \mathbb R^n$. In the next proposition we give an explicit expression of this differential operator.
\begin{prop}
 For $s,t\in \mathbb C$, let
\begin{gather*}
 F_{s,t} := \vert x-y\vert^2 \Delta_x\otimes \Delta_y
- 2(s+n+1) \sum_{j=1}^n (x_j-y_j) \frac{\partial}{\partial x_j} \otimes \Delta_y
\\ \hphantom{ F_{s,t} :=}
{}-2(t+n+1) \sum_{j=1}^n (y_j-x_j) \Delta_x\otimes \frac{\partial}{\partial y_j}
-2\rho(x-y)\fsl{D}_x\otimes \Delta_y\ -\ 2 \Delta_x\otimes \rho'(y-x) \fsl{D}_y
\\ \hphantom{ F_{s,t} :=}
{}+(t+1)(t+n+1) \Delta_x\otimes \id\ +\ (s+1)(s+n+1)\id\otimes\Delta_y
\\ \hphantom{ F_{s,t} :=}
{}-2(s+n+1)(t+n+1) \sum_{j=1}^n \frac{\partial}{\partial x_j} \otimes \frac{\partial}{\partial y_j}
-2(s+n+1) \sum_{j=1}^n \frac{\partial}{\partial x_j}\otimes \rho'(e_j)\fsl{D}'_y
\\ \hphantom{ F_{s,t} :=}
{}-2(t+n+1) \sum_{j=1}^n \rho(e_j)\fsl{D}_x\otimes \frac{\partial}{\partial y_j}
- 2\bigg(\sum_{j=1}^n \rho(e_j) \fsl{D}_x\otimes \rho'(e_j)\fsl{D}'_y\bigg).
\end{gather*}
Then, the symbol of the differential operator $F_{s,t}$ is equal to $f_{s,t}$ given in Proposition~$\ref{symb-F-s-t}$.
\end{prop}

The symbol calculus yields also the following theorem, which is the main formula leading to the proof of Theorem~\ref{maintheorem}.
\begin{theor} The following identity holds for $s,t\in \mathbb C$
\begin{gather}\label{estrr}
\big(\fsl{\mathcal R}_s\otimes \fsl{\mathcal R}'_t\big)\circ \mathcal M =c(s,t)\ F_{s,t}\circ \big(\fsl{\mathcal R}_{s+2} \otimes \fsl{\mathcal R}'_{t+2}\big),
\end{gather}
where
\begin{gather*}
c(s,t) = \frac{1}{(s+1)(s+n+1)(t+1)(t+n+1)}.
\end{gather*}
\end{theor}
\begin{proof} The identity for the symbols obtained in Proposition~\ref{symb-F-s-t} is translated as
\begin{gather*}
\fsl{c}_s^{-1} \fsl{c}_t^{-1}\big(\fsl{\mathcal R}_s\otimes \fsl{\mathcal R}_t'\big)\circ \mathcal M = \fsl{c}_{s+2}^{-1}\,\fsl{c}_{t+2}^{-1}\ F_{s,t}\circ\big(\fsl{\mathcal R}_{s+2}\otimes \fsl{\mathcal R}_{t+2}'\big).
\end{gather*}
An elementary computation gives
\begin{gather*}
\frac{\fsl{c}_{s+2}}{\fsl{c}_s} = -(s+1)(s+n+1)
\end{gather*}
and the theorem follows.
\end{proof}

\subsection{The proof of the main theorem}
We now study the behaviour of the operators involved in the previous construction under the action of the conformal group $\mathbf G$.

The main observation is that up to a shift in the parameters, and up to a constant multiple, the operators $\fsl{\mathcal R}_s$ and $\fsl{\mathcal R}'_t$ are essentially the Knapp--Stein operators considered in Section~\ref{S33}.

For $\lambda, \mu\in \mathbb C$ generic, compare~\eqref{defKS} and~\eqref{defCR} to get
\begin{gather*}
I_\lambda = \fsl{c}_{2\lambda- 2n} \fsl{\mathcal R}_{2\lambda- 2n}, \qquad I'_\mu = \fsl{c}_{2\mu-2n} \fsl{\mathcal R}'_{2\mu-2n}. \end{gather*}
 Change the normalization of the Knapp--Stein operator, that is, redefine the
 Knapp--Stein ope\-ra\-tors by setting
 \begin{gather*}
 \widetilde I_\lambda = \fsl{\mathcal R}_{2\lambda-2n},\qquad \widetilde I'_\mu = \fsl{\mathcal R}'_{2\mu-2n}.
 \end{gather*}
 Moreover, set
 \begin{gather*}
 s= -2\lambda-2,\qquad t=-2\mu-2
 \end{gather*}
 and
 \begin{gather*}
 E_{\lambda,\mu} = F_{s,t} =F_{-2\lambda-2, -2\mu-2}.
 \end{gather*}
Notice that
\begin{gather*}
s+n+1 =-2\lambda+n-1, \qquad
s+1 = -2\lambda-1,
\end{gather*}
so that
\begin{gather}
E_{\lambda,\mu} = \vert x-y\vert^2 \Delta_x\otimes \Delta_y
+2(2\lambda-n+1) \sum_{j=1}^n (x_j-y_j) \frac{\partial}{\partial x_j} \otimes \Delta_y\nonumber
\\ \hphantom{E_{\lambda,\mu} =}
{}+ 2(2\mu-n+1) \sum_{j=1}^n (y_j-x_j) \Delta_x\otimes \frac{\partial}{\partial y_j}
-2\rho(x-y)\fsl{D}_x\otimes \Delta_y - 2 \Delta_x\otimes \rho'(x-y) \fsl{D}_y\nonumber
\\ \hphantom{E_{\lambda,\mu} =}
{}+(2\mu-n+1)(2\mu+1) \Delta_x\otimes \id + (2\lambda-n+1) (2\lambda+1)\id\otimes\Delta_y\nonumber
\\ \hphantom{E_{\lambda,\mu} =}
{} -2(2\lambda-n+1)(2\mu-n+1) \sum_{j=1}^n \frac{\partial}{\partial x_j} \otimes \frac{\partial}{\partial y_j}
+2(2\lambda-n+1) \sum_{j=1}^n \frac{\partial}{\partial x_j}\otimes \rho'(e_j)\fsl{D}'_y\nonumber
\\ \hphantom{E_{\lambda,\mu} =}
{}+2(2\mu-n+1)\sum_{j=1}^n \rho(e_j)\fsl{D}_x\otimes \frac{\partial}{\partial y_j}
- 2\bigg(\sum_{j=1}^n \rho(e_j)\fsl{D}_x\otimes \rho'(e_j)\fsl{D}'_y\bigg).\label{Elambdamu}
\end{gather}
 Now~\eqref{estrr} can be rewritten as
\begin{gather*}
\big(\widetilde I_{n-\lambda-1} \otimes \widetilde I'_{n-\mu-1}\big) \circ \mathcal{M} = d(\lambda, \mu)\, E_{\lambda, \mu}\circ \big(\widetilde I_{n-\lambda}\otimes \widetilde I'_{n-\mu}\big),
\end{gather*}
where
\begin{gather*}
d(\lambda, \mu) = \frac{1}{(2\lambda-n+1)(2\lambda+1)(2\mu-n+1)(2\mu+1)}.
\end{gather*}
\begin{theor}\label{covE1th}
 The differential operator $E_{\lambda,\mu}$ satisfies, for any $g\in \mathbf G$
\begin{gather*}\label{covE1}
 E_{\lambda,\mu} \circ \big(\pi_\lambda(g) \otimes \pi'_\mu(g)\big)=\big( \pi_{\lambda+1}(g)\otimes \pi'_{\mu+1}(g)\big)\circ E_{\lambda,\mu}.
\end{gather*}
The equality holds when applied to functions $f \in C^\infty(\mathbb R^n \times \mathbb R^n, \mathbb S \otimes \mathbb S')$ with compact support and such that the action of $g$ is defined on the support of $f$.
\end{theor}

\begin{proof}
As $\mathbf G$ is connected, Theorem~\ref{covE1th} is equivalent to its infinitesimal version, which we now formulate.

\begin{theor}\label{thcovE2}
 For any $X\in \mathfrak g$,
\begin{gather}\label{covE2}
E_{\lambda, \mu} \circ \big({{\rm d}}\pi_\lambda(X)\otimes \id+\id\otimes\, {{\rm d}}\pi'_\mu(X)\big)= \big({{\rm d}}\pi_{\lambda+1}(X)\otimes \id+\id\otimes\, {{\rm d}}\pi'_{\mu+1}(X)\big)\circ E_{\lambda, \mu}.
\end{gather}
\end{theor}
A well-known and easy-to-prove result is that ${{\rm d}}\pi_\lambda(X)$ is a differential operator with $\End(\mathbb S)$-valued polynomial coefficients, hence preserves the space $\mathcal S(\mathbb R^n, \mathbb S)$, so that both sides of~\eqref{covE2} are well defined and are differential operators on $\mathbb R^n\times \mathbb R^n$ with $\End(\mathbb S \otimes \mathbb S')$-valued polynomial coefficients.

 In order to prove Theorem~\ref{thcovE2}, let for $X\in \mathfrak g$
 \begin{gather*}\label{covE3}
 A_{\lambda,\mu}(X) = E_{\lambda, \mu} \circ \big({{\rm d}}\pi_\lambda(X)\otimes \id\!+\id\otimes\, {{\rm d}}\pi'_\mu(X)\big)\!- \big({{\rm d}}\pi_{\lambda+1}(X)\otimes \id\!+\id\otimes\, {{\rm d}}\pi'_{\mu+1}(X)\big)\circ E_{\lambda, \mu}.
 \end{gather*}
 We want to prove that $A_{\lambda, \mu}(X)=0$ for any $X\in \mathfrak g$, and in order to do it, we first prove the following weaker statement.
 \begin{lem} For any $X\in \mathfrak g$
 \begin{gather}\label{covE3bis}
 A_{\lambda, \mu} (X)\circ \big(\widetilde I_{n-\lambda} \otimes \widetilde I'_{n-\mu}\big)=0.
 \end{gather}
 \end{lem}
\begin{proof}
It is sufficient to prove the results for $(\lambda, \mu)$ generic, so that we the may assume that $\lambda$, $\lambda+1$, $n-\lambda$, $n-\lambda-1$ are not poles of $\widetilde I_\lambda$ and same conditions on $\mu$. Also assume that $(\lambda, \mu)$ is not a pole of the rational function $d(\lambda, \mu)$. Thus for any $X\in \mathfrak g$
\begin{gather*}
\big({{\rm d}}\pi_{\lambda+1}(X)\otimes \id+\id\otimes\, {{\rm d}}\pi'_{\mu+1}(X)\big)\circ d(\lambda,\mu) E_{\lambda, \mu}\circ \big(\widetilde I_{n-\lambda}\otimes \widetilde I'_{n-\mu}\big)
\\ \qquad
{}= \big({{\rm d}}\pi_{\lambda+1}(X)\otimes \id+\id\otimes\, {{\rm d}}\pi'_{\mu+1}(X)\big)\circ \big(\widetilde I_{n-\lambda-1} \otimes \widetilde I'_{n-\mu-1} \big)\circ \mathcal{M}
\\ \qquad
{}= \big(\widetilde I_{n-\lambda-1} \otimes \widetilde I'_{n-\mu-1} \big)\circ \big({{\rm d}}\pi_{n-\lambda-1}(X)\otimes \id+\id\otimes\, {{\rm d}}\pi'_{n-\mu-1}(X)\big)\circ \mathcal{M}
\\ \qquad
{}=\big(\widetilde I_{n-\lambda-1} \otimes \widetilde I'_{n-\mu-1} \big)\circ \mathcal{M}\circ\big({{\rm d}}\pi_{n-\lambda}(X)\otimes \id+\id\otimes\,{{\rm d}}\pi'_{n-\mu}(X)\big)
\\ \qquad
{}= d(\lambda,\mu) E_{\lambda,\mu} \circ \big(\widetilde I_{n-\lambda} \otimes \widetilde I'_{n-\mu} \big)\circ \big({{\rm d}}\pi_{n-\lambda}(X)\otimes \id+\id\otimes\, {{\rm d}}\pi'_{n-\mu}(X)\big) \big)
\\ \qquad
{}= d(\lambda,\mu) E_{\lambda, \mu} \circ \big({{\rm d}}\pi_{\lambda}(X)\otimes \id+\id\otimes\, {{\rm d}}\pi'_{\mu}(X)\big)\circ \big(\widetilde I_{n-\lambda} \otimes \widetilde I'_{n-\mu} \big) ,
\end{gather*}
and~\eqref{covE3bis} follows.
\end{proof}

The proof of Theorem~\ref{thcovE2} is achieved through the following lemma, valid in a more general context.

\begin{lem} Let $V$ be a finite-dimensional vector space. Let $D$ be a differential operator acting on $C^\infty(\mathbb R^p, V)$ with $\End(V)$-valued polynomial coefficients. Let $K$ be a convolution operator on~$\mathbb R^p$ by an $\End(V)$-valued tempered distribution $k$. Assume that its Fourier transform $\widehat k$ coincides on a dense open subset $\mathcal O\subset \mathbb R^p$ with an $\End(V)$-valued smooth function and satisfies
\begin{gather*}
\text{for any}\quad \xi\in \mathcal O \qquad \widehat{k}(\xi) \in {\rm GL}(V).
\end{gather*}
Assume further that $D\circ K=0$. Then $D=0$.
\end{lem}
\begin{proof} Under the Fourier transform, the operator $K$ corresponds to the multiplication operator by $\widehat k$, and the operator $D$ corresponds to a differential operator \begin{gather*}\widehat D= \sum_I a_I(\xi) \partial_\xi^I\end{gather*}
 on $\mathbb R^p$ with $\End(V)$-valued polynomial coefficients. The assumption $D\circ K=0$ implies $\widehat D \circ \widehat K = 0$, or in other words $\widehat D\big(\widehat k \psi\big) = 0$ for any function $\psi\in \mathcal S(\mathbb R^p, V)$.

Let $\xi_0\in \mathcal O$, $v_0\in V$ and $I_0$ a $p$-multi-index. There exists a smooth $V$-valued function $\varphi_0$ with compact support included in $\mathcal O$ and such that in a neighbourhood of $\xi_0$
\begin{gather*}
\varphi_0(\xi) =\frac{1}{I_0!} (\xi-\xi_0)^{I_0} v_0,
\end{gather*}
so that $\partial^{I_0}\varphi_0(\xi_0) = v_0$. Now let $\psi_0$ be defined on $\mathcal O$ by
\begin{gather*}
\psi_0(\xi) = \widehat k (\xi) ^{-1} \varphi_0(\xi)
\end{gather*}
and equal to $0$ outside of $\mathcal O$. The function $\psi$ is a smooth function with compact support on $\mathbb R^p$ and
\begin{gather*}
0=\widehat D \big(\widehat k \psi_0\big) (\xi_0) = \widehat D(\varphi_0)(\xi_0) = a_{I_0}(\xi_0) v_0.
\end{gather*}
This being valid for any $v_0\in V$, it follows that $a_{I_0}(\xi_0) = 0$. As $\xi_0$ was arbitrary in $\mathcal O$ and $a_{I_0}$ is a polynomial, this implies $a_{I_0}\equiv 0$ and finally, as $I_0$ was arbitrary $\widehat D=0$. This finishes the proof of the lemma.
\end{proof}

For generic $\lambda$, $\mu$, the operator $K=\widetilde I_{n-\lambda}\otimes \widetilde I'_{n-\mu}$ satisfies the conditions of the lemma. Hence~\eqref{covE2} holds true and Theorem~\ref{covE1th} follows.
\end{proof}

\section[The symmetry breaking differential operators for the tensor product of two spinorial representations]
{The symmetry breaking differential operators\\ for the tensor product of two spinorial representations}

\subsection[The projections tilde Psi(k)]{The projections $\boldsymbol{\widetilde \Psi^{(k)}}$}

Recall the study of the tensor product $\mathbb S\otimes \mathbb S'$ under the action of $\mathbf{M} = \Spin(n)$ and in particular
for $k$, $1\leq k\leq n$, there is an $\mathbf M$-intertwining map $\Psi^{(k)} \colon \mathbb S\otimes \mathbb{S}' \to \Lambda^*_k(\mathbb{R}^n)\otimes \mathbb C$. Recall that $\tau^*_k$ is the representation of $\bf M$ on $\Lambda^*_k(\mathbb{R}^n)\otimes\mathbb{C}$.

For $\nu\in \mathbb C$ and $k$, $0\leq k\leq n$, let
 \begin{gather*}
 \pi_{k;\nu} =\Ind_\mathbf P^\mathbf G{ \tau^*_k}\otimes \chi_\nu \otimes 1\end{gather*}
 which is realized on (a subspace of) the space $C^\infty(\mathbb R^n, \Lambda^*_k(\mathbb{R}^n)\otimes \mathbb C)$ in the non-compact picture. Further let
\begin{gather*}
\widetilde \Psi^{(k)} \colon\quad C^\infty(\mathbb R^n\times \mathbb R^n,
\mathbb S\otimes \mathbb{S}' )\to C^\infty( \mathbb R^n,
\Lambda^*_k\big(\mathbb R^n)\otimes\mathbb C\big),\qquad
F\mapsto \big(\Psi^{(k)}F(x,y)\big)_{\vert x=y}.
\end{gather*}
Let $\lambda, \mu \in \mathbb C$. Form the spinorial representations
\begin{gather*}
\pi_\lambda= \Ind_\mathbf P^\mathbf G \rho\otimes \chi_\lambda \otimes 1, \qquad \pi'_\mu= \Ind_\mathbf P^\mathbf G \rho'\otimes \chi_\mu\otimes 1
\end{gather*}
and the tensor product $\pi_\lambda \otimes \pi'_\mu$. The following result is a consequence of the functoriality of the induction process.

\begin{prop} The map $\widetilde \Psi^{(k)}$ satisfies
\begin{gather*}
\widetilde\Psi^{(k)}\circ (\pi_\lambda(g) \otimes \pi'_\mu(g)) = \pi_{k;\lambda+\mu}(g) \circ \widetilde\Psi^{(k)}.
\end{gather*}
\end{prop}

\subsection{Definition of the SBDO}
For $m\in \mathbb N$, define the operator $E_{\lambda, \mu}^{(m)} \colon C^\infty(\mathbb R^n \times \mathbb R^n, \mathbb S\otimes \mathbb S')\to C^\infty(\mathbb R^n \times \mathbb R^n, \mathbb S\otimes \mathbb S')$ by
\begin{gather*}
E_{\lambda, \mu}^{(m)}=E_{\lambda+m-1, \mu+m-1}\circ \dots \circ E_{\lambda, \mu}.
\end{gather*}
The operator $E_{\lambda, \mu}^{(m)}$ satisfies, for any $g\in \bf G$
\begin{gather}\label{covEm}
E_{\lambda, \mu}^{(m)}\circ \left(\pi_\lambda(g) \otimes \pi'_\mu(g) \right)= \left(\pi_{\lambda+m}(g) \otimes \pi'_{\mu+m}(g)\right)\circ E_{\lambda, \mu}^{(m)}.
\end{gather}
Let
\begin{gather*}
B_{k;\lambda, \mu}^{(m)}= \widetilde \Psi^{(k)} \circ E_{\lambda, \mu}^{(m)}.
\end{gather*}

\begin{prop} \quad
\begin{enumerate}\itemsep=0pt
\item[$(i)$] The operators $B_{k;\lambda, \mu}^{(m)} \colon C^\infty\big(\mathbb R^n\times \mathbb R^n, \mathbb S\otimes \mathbb S'\big)\rightarrow C^\infty(\mathbb R^n, \Lambda^*_k(\mathbb R^n)\otimes \nobreak \mathbb C)$ are constant coefficient bi-differential operators and homogeneous of degree $2m$.

\item[$(ii)$] For any $g\in \mathbf G$
\begin{gather*}
B_{k;\lambda, \mu}^{(m)}\circ \big(\pi_\lambda(g) \otimes \pi'_\mu(g)\big)= \pi_{k;\lambda+\mu+2m}(g)\circ B_{k;\lambda, \mu}^{(m)}.
\end{gather*}
\end{enumerate}
\end{prop}
\begin{proof}
$(ii)$ is a direct consequence of the covariance property of the source operators and of the map $\widetilde \Psi^{(k)}$.

Next apply $(ii)$ to the case where $g$ is a translation by an element of $\mathbb R^n$. This implies that~$B^{(m)}_{k;\lambda,\mu}$ commutes with (diagonal) translations and hence has constant coefficients. Apply then to the case where $g$ belongs to $A$ acting by dilations of $\mathbb R^n$ to get the homogeneity of degree~$2m$ for the operator $B^{(m)}_{k;\lambda,\mu}$. This completes the proof of $(i)$.
\end{proof}

The definition of the SBDO $B_{k;\lambda, \mu}^{(m)}$ yields a recurrence formula for these operators. Use the covariance relation~\eqref{covEm} applied to diagonal translations on $\mathbb R^n\times \mathbb R^n$ to see that the coefficients of $E_{\lambda,\mu}^{(m)}$ are (operators valued)-functions of $(x-y)$. Let $\overset{o}{E}^{(m)}_{\lambda,\mu}$ be the constant coefficients part of $E_{\lambda,\mu}^{(m)}$.

\begin{prop} The SBDO $B_{k;\lambda,\mu}^{(m)}$ satisfies the recurrence relation
\begin{gather*}
B_{k;\lambda,\mu}^{(m)} = B_{k;\lambda+1,\mu+1}^{(m-1)}\circ E_{\lambda,\mu} .
\end{gather*}
\end{prop}

\begin{proof}

All coefficients of the difference $E_{\lambda,\mu}^{(m)}-\overset{o}{E}^{(m)}_{\lambda,\mu}$ vanish on the diagonal of $\mathbb R^n\times \mathbb R^n$. Hence
\begin{gather*}
\widetilde \Psi^{(k)} \circ E_{\lambda,\mu}^{(m)} = \widetilde \Psi^{(k)} \circ \overset{o}{E}_{\lambda,\mu}^{(m)}.
\end{gather*}
Now
\begin{gather*}
E_{\lambda,\mu}^{(m)} = \left(E_{\lambda+m-1, \mu+m-1}\circ \dots \circ E_{\lambda+1,\mu+1} \right)\circ E_{\lambda,\mu}
= E_{\lambda+1,\mu+1}^{(m-1)} \circ E_{\lambda, \mu}.
\end{gather*}
Hence
\begin{gather*}
B_{k;\lambda, \mu}^{(m)} = \widetilde \Psi^{(k)} \circ E_{\lambda+1, \mu+1}^{(m-1)}\circ E_{\lambda,\mu} = B_{k;\lambda+1,\mu+1}^{(m-1)}\circ E_{\lambda,\mu}.\tag*{\qed}
\end{gather*}
\renewcommand{\qed}{}
\end{proof}

\subsection{An example}

Let us write explicitly the SBDO for the case $k=0$ and $m=1$.

\begin{theor}
The operator
\begin{gather*}
B_{0;\lambda, \mu}^{(1)} \colon\ C^\infty(\mathbb R^n\times \mathbb R^n, \mathbb S\otimes \mathbb S') \to C^\infty(\mathbb R^n)
\end{gather*}
is given by
\begin{gather}
B_{0;\lambda, \mu}^{(1)}\big(v(\cdot)\otimes w'(\cdot)\big)(x) =
(2\mu-n+1) (2\mu+1)\big( \Delta v(x), w'(x)\big) \nonumber
\\ \hphantom{B_{0;\lambda, \mu}^{(1)}\big(v(\cdot)\otimes w'(\cdot)\big)(x) =}
{}+(2\lambda-n+1)(2\lambda+1)\big(v(x), \Delta w'(x)\big)\nonumber
\\ \hphantom{B_{0;\lambda, \mu}^{(1)}\big(v(\cdot)\otimes w'(\cdot)\big)(x) =}
{} -2(2\lambda-n+1)(2\mu-n+1)
 +\sum_{j=1}^n \bigg(\frac{\partial}{\partial x_j}v(x), \frac{\partial}{\partial y_j}w'(x)\bigg)\nonumber
 \\ \hphantom{B_{0;\lambda, \mu}^{(1)}\big(v(\cdot)\otimes w'(\cdot)\big)(x) =}
{}-2(2\lambda+2\mu-n+2)\big( \fsl{D}v(x), \fsl{D}'w'(x)\big).\label{B_0}
\end{gather}
The operator $B_{0;\lambda, \mu}^{(1)}$ satisfies, for any $g\in \mathbf G$
\begin{gather*}
B_{0;\lambda, \mu}^{(1)}\circ \big(\pi_\lambda(g) \otimes \pi_\mu'(g)\big) = \pi_{0;\lambda+\mu+2}(g) \circ B_{0;\lambda, \mu}^{(1)} .
\end{gather*}
\end{theor}
\begin{proof}
First notice that $\Psi^{(0)}\colon \mathbb S \otimes \mathbb S'\to \mathbb C$ is given by
\begin{gather*}
\Psi^{(0)} (v\otimes w') = (v,w'),
\end{gather*}
so that $\widetilde \Psi^{(0)} \colon C^\infty(\mathbb R^n \times \mathbb R^n,\mathbb S\otimes \mathbb S')\to C^\infty(\mathbb R^n)$, is given by
\begin{gather*}
\widetilde \Psi^{(0)}\big(v(\cdot)\otimes w'(\cdot)\big)(x) = \big(v(x),w'(x)\big) .
\end{gather*}
Now use~\eqref{Elambdamu} and observe that by~\eqref{1rhorho'1}
\begin{gather*}
\bigg(\frac{\partial}{\partial x_j} v(x), \rho'(e_j) \fsl{D}'w'(x) \bigg) = -\bigg(\rho(e_j) \frac{\partial}{\partial x_j} v(x), \fsl{D}'w'(x)\bigg)
\end{gather*}
so that
\begin{gather*}
\sum_{j=1}^n \bigg(\frac{\partial}{\partial x_j} v(x), \rho'(e_j) \fsl{D}'w'(x) \bigg) = -\big( \fsl{D} v(x), \fsl{D}'w'(x)\big),
\end{gather*}
and a similar result holds for $\sum_{j=1}^n \big( \rho(e_j) \fsl{D}v(x), \frac{\partial}{\partial y_j} w'(x) \big) $. Also use~\eqref{rhorho'3} for $k=0$ to get
\begin{gather*}
\Psi^{(0)}\bigg(\sum_{j=1}^n \big(\rho(e_j) \fsl{D}_x v(\cdot)\otimes \rho'(e_j) \fsl{D}'_y w'(\cdot) \big)\bigg)(x)=n \big( \fsl{D}_x v(x), \fsl{D}'_y w'(x)\big).
\end{gather*}
The final expression for $B_{0;\lambda,\mu}^{(1)}$ is obtained by putting together the partial computations.
\end{proof}

\subsection[The dimension n=1 and the classical Rankin--Cohen brackets]
{The dimension $\boldsymbol{n=1}$ and the classical Rankin--Cohen brackets}

Let $E=\mathbb R$ be the standard Euclidean space of dimension $n=1$ and denote by $e$ be the vector $1$ (to distinguish it from the scalar $1$). Let $\cl(E)$ the corresponding Clifford algebra which is isomorphic to the complex plane. Let $\ccl{(E)}$ be its complexification. The spin group $\Spin(E)$ is equal to $\{1,-1\}$.

Let $\mathbb S = \mathbb C$ and define for $v\in \mathbb S$ and $x\in \mathbb R$
\begin{gather*}
\rho(xe) v = {\rm i}\,xv
\end{gather*}
and extend it as an action of $\ccl(E)$ on $\mathbb S$, still denoted by $\rho$.
Similarly, let $\mathbb S'=\mathbb C$ and define for $w'\in \mathbb S'$ and $x \in \mathbb S'$
\begin{gather*}
\rho'(xe) w' = -{\rm i}\, xw'.
\end{gather*}
Through the duality on $(\mathbb S, \mathbb S') $ given by $(v,w')\mapsto vw'$, $(\rho',\mathbb S')$ is the dual Clifford module of~$(\rho, \mathbb S)$, i.e., for $x\in \mathbb R$ and $v\in \mathbb S, w'\in \mathbb S'$
\begin{gather*}
(\rho(x)v,w') = -(v,\rho'(x)w').
\end{gather*}
The corresponding Dirac operators are
\begin{gather*}
\fsl{D} = {\rm i} \frac{{\rm d}}{{\rm d} x}\,,\qquad
\fsl{D'} = -{\rm i} \frac{{\rm d}}{{\rm d} x}\, .
\end{gather*}
For a smooth $\mathbb S$-valued (resp.~$\mathbb S'$-valued) function $v(x)$ (resp.~$w'(x)$),
\begin{gather*}
\big(\fsl{D} v(x), \fsl{D'} w'(y)\big)= \frac{{\rm d}v}{{\rm d}x} (x)\,\frac{{\rm d}w'}{{\rm d}y}(y).
\end{gather*}
By substituting these results in~\eqref{B_0}, one obtains
\begin{gather}\label{B_0dim1}
B_{0;\lambda, \mu}^{(1)} = 2\mu(2\mu+1)\frac{\partial^2}{\partial x^2}+2\lambda(2\lambda+1) \frac{\partial^2}{\partial y^2}-2(2\lambda+1)(2\mu+1) \frac{\partial^2}{\partial x\partial y}
\end{gather}
and this coincides (up to a constant multiple) to the degree two Rankin--Cohen operator for the group $\mathrm{SL}(2,\mathbb R)$ which is isomorphic to $\Spin_0(1,2)$. See \cite[Theorem~10.7]{bck1} and~\cite{kp16} for more general results in this direction.

\subsection*{Acknowledgements}
 At the very beginning of the present work, the first author benefited from a discussion with Bent {\O}rsted during a visit at Aarhus University and wishes to thank him and its institution for the invitation. The authors are very grateful to the anonymous referees for their expert comments and suggestions which helped to improve the initial version of this article.

\pdfbookmark[1]{References}{ref}
\LastPageEnding
\end{document}